\documentclass{amsart}
\usepackage[all]{xy}
\usepackage{amsmath,amsfonts,amssymb,amsthm,epsfig,amscd,comment,latexsym,psfrag}
\usepackage{nicefrac,xspace,tikz}
\usepackage{pb-diagram}
\usepackage{pifont}
\usetikzlibrary{arrows}

\title{The Soergel category and the redotted Webster algebra}
\author{Mikhail Khovanov}
\address{Department of Mathematics \\ Columbia University \\ New York, United States}
\email{khovanov@math.columbia.edu}
\author{Joshua Sussan}
\address{Department of Mathematics\\ CUNY Medgar Evers \\ New York, United States}
\email{jsussan@mec.cuny.edu}
\date{May 9, 2016}

\newtheorem{prop}{Proposition}

\newtheorem{ex}{Example}
\newtheorem{theorem}{Theorem}
\newtheorem{corollary}{Corollary} 
\newtheorem{lemma}{Lemma}
\newtheorem{remark}{Remark}
\newtheorem*{thm}{Theorem}
\newcommand{\oplusop}[1]{{\mathop{\oplus}\limits_{#1}}}
 \newcommand{\oplusoop}[2]{{\mathop{\oplus}\limits_{#1}^{#2}}}
 
\begin{document} 

\maketitle
\baselineskip 14pt

\def\JS#1{\textcolor[rgb]{0.40,0.00,0.90}{[JS: #1]}}%
\def\R{\mathfrak{R}}
\def\r{\mathfrak{r}}
\def\b{\mathfrak{b}}
\def\Q{\mathbb Q}
\def\Z{\mathbb Z}
\def\N{\mathbb N} 
\def\C{\mathbb C}
\def\F{\mathbb F}
\def\{{\lbrace}
\def\}{\rbrace}
\def\o{\otimes}
\def\lra{\longrightarrow}
\def\Hom{\mathrm{Hom}}
\def\End{\mathrm{End}}
\def\HH{\mathrm{HH}}
\def\HOM{\mathrm{HOM}}
\def\RHom{\mathrm{RHom}}
\def\Id{\mathrm{Id}}
\def\mc{\mathcal}
\def\mf{\mathfrak} 
\def\Ext{\mathrm{Ext}}
\def\Ind{\mathrm{Ind}}
\def\Res{\mathrm{Res}}
\def\soc{\mathrm{soc}}
\def\hd{\mathrm{hd}}
\def\Seq{\mathrm{Seq}}
\def\gmod{\mathrm{gmod}}
\def\-gmod{\mathrm{-gmod}}
\def\gpmod{\mathrm{-gpmod}}

\def\gdim{\mathrm{gdim}}
\def\shuffle{\,\raise 1pt\hbox{$\scriptscriptstyle\cup{\mskip
               -4mu}\cup$}\,}
\newcommand{\define}{\stackrel{\mbox{\scriptsize{def}}}{=}}


%
\def\drawing#1{\begin{center}\epsfig{file=#1}\end{center}}

 \def\yesnocases#1#2#3#4{\left\{
\begin{array}{ll} #1 & #2 \\ #3 & #4
\end{array} \right. }

\newcommand{\LOT}{H^-} 

\begin{abstract} 
We describe a collection of graded rings which surject onto Webster rings for sl(2) and which should be related to certain categories of singular Soergel bimodules.  In the first non-trivial case, we construct a categorical braid group action which categorifies the Burau representation.
\end{abstract}

\tableofcontents 


\section{Introduction}



For a finite Coxeter system $(W,S)$, Soergel introduced a category of special bimodules $ \mathcal{R}$, over polynomials rings, now known as Soergel bimodules ~\cite{Soe}.  
He parameterized the indecomposable objects and showed that $\mathcal{R}$ categorifies the Hecke algebra for $(W,S)$. 
We will only consider the case when $W=S_n$ and $S$ is the set of simple transpositions.
This category has a graphical description given in ~\cite{EK}.  

The category of singular Soergel bimodules ${}^I \mathcal{R}^J$ (bimodules over certain invariant subrings) is not as well studied.  Williamson ~\cite{Wi} parameterized the indecomposable objects of ${}^I \mathcal{R}^J$.  The special case of $I=\emptyset$ and $J=\{1,\ldots,k-1,k+1,\ldots,n-1 \}$ is of interest to us because it should categorify a weight space of a tensor product of standard representations of quantum $\mathfrak{sl}_2$.
In a future work we hope to endow this category with a p-DG structure and show its relationship to categorification of small quantum groups.

Here we suggest a graphical approach to singular Soergel bimodule categories.  We prove that our diagrammatics in a further specialization does describe the category of singular Soergel bimodules ${}^{\emptyset} \mathcal{R}^{\{2,\ldots,n-1\}}$.

In Section ~\ref{sectionW(n,k)} we define rings $W(n,k)$ algebraically and graphically.  We expect that the category of graded finitely-generated projective $W(n,k)$-modules is equivalent to the category of singular Soergel bimodules $  {}^{\emptyset} \mathcal{R}^{\hat{k}}= {}^{\emptyset} \mathcal{R}^{\{1,\ldots, k-1,k+1,\ldots,n-1     \}}$.
In Section ~\ref{sectionW(n,1)} we specialize to the rings $W(n,1)$ and compute various features of these rings.  
We define the category of singular Soergel bimodules in Section ~\ref{sectionsingsoerg}  and explicitly construct the indecomposable objects for our special case $k=1$.
The diagrammatic Hecke category is reviewed in Section ~\ref{sectionhecke}.
In Section ~\ref{sectionequivalence} we prove one of our main results.
\begin{thm}
The category of finitely generated graded projective modules
$W(n,1) \gpmod$ is equivalent to the category of singular Soergel bimodules $  {}^{\emptyset} \mathcal{R}^{\hat{1}}$. 
\end{thm}

In Section ~\ref{sectiondeformation} we consider an algebra $A_n^!$ which appears in representation theory and symplectic topology.  We compute its Hochschild cohomology and prove the following result.
\begin{thm}
$W(n,1)$ is a deformation of  $A_n^!$.  More precisely, $\Bbbk[x_1, \ldots, x_n]$ is a subalgebra of $W(n,1)$ and if $\mathfrak{m}$ is the ideal of $\Bbbk[x_1, \ldots, x_n]$ of polynomials containing a non-constant term, then there is an isomorphism of algebras
\begin{equation*}
W(n,1)/ \mathfrak{m} \cong A_n^!.
\end{equation*}
\end{thm}

There is a well-known action of the braid group on the homotopy category of regular Soergel bimodules ~\cite{Ro2}.  This construction could be extended to a categorification of the HOMFLYPT polynomial ~\cite{Kh2, Ro3}.  In fact, there is a functor from the category of braid cobordisms into the homotopy category of regular Soergel bimodules ~\cite{EK}.
In Section ~\ref{sectionbraid} we adapt this categorical braid group action to the algebra $W(n,1)$.
Our third main result is the following.
\begin{thm}
There are endofunctors $\sigma_i $ and $\sigma_i^{-1}$ of the homotopy category
of graded finitely-generated projective $W(n,1)$-modules, satisfying the following isomorphisms
\begin{enumerate}
\item $ \sigma_i \circ \sigma_i^{-1} \cong \Id \cong \sigma_i^{-1} \circ \sigma_i $,
\item $\sigma_i \circ \sigma_j \cong \sigma_j \circ \sigma_i$ if $|i-j|>1$,
\item $ \sigma_i \circ \sigma_{i+1} \circ \sigma_i \cong \sigma_{i+1} \circ \sigma_i \circ \sigma_{i+1}$.
\end{enumerate}
\end{thm}
The functors $\sigma_i^{\pm 1}$ are given by tensoring with complexes of certain $W(n,1)$-bimodules.  In order to prove the theorem above we construct a functor from the (diagrammatic) category of Soergel bimodules to a diagrammatic tensor category $\mathcal{B}$ that we introduce in Section ~\ref{sectionbraid}.  The main result of ~\cite{EK} then implies the theorem above.  We expect that the theorem extends to algebras $W(n,k)$ for all $k$.


\subsection{Acknowledgements}
M.K. was supported by NSF grant DMS-1406065.
J.S. was supported by NSF grant DMS-1407394, PSC-CUNY Award 67144-00 45, and an Alfred P. Sloan Foundation CUNY Junior Faculty Award.

\section{Algebras $W(n,k)$} 
\label{sectionW(n,k)}
\subsection{Conventions}
Throughout we fix an infinite field $\Bbbk$ of characteristic not equal to $2$.

For graded algebras $A$ and $B$ we denote by $A \-gmod$ the category of graded $A$-modules and $(A,B)\-gmod$ the category of graded $(A,B)$-bimodules.
Let $A\gpmod$ be the category of finitely generated graded projective $A$-modules. 

Let $R_n=\Bbbk[x_1, \ldots, x_n]$.  When the context is clear we will often denote $R_n$ simply by $R$.
We will write $E_i(x_1, \ldots, x_n)$ for the $i$th elementary symmetric polynomial in the variables $x_1, \ldots, x_n$.  When there is no chance for confusion we will abbreviate it as $E_i$.  $R_n$ is a graded algebra where each $x_i$ has degree two.

For a graded $A$-module $M$ we denote the subspace in degree $i$ by $M_i$.
Let $M \langle r \rangle$ denote the grading shift of the module $M$ up by $r$.
Explicitly $(M \langle r \rangle)_i=M_{i-r}$.

\begin{remark}
Our grading shift convention is opposite to the ones used in ~\cite{Soe} and ~\cite{Wi} but consistent with the one used in ~\cite{EK}.
\end{remark}

\subsection{The algebra $ W(n,k)$}
\label{algebrarelationsI}
Let $ \textrm{Seq} $ be the set of all sequences $ {\bf i}=(i_1, \ldots, i_{n+k}) $, where $ i_{j} \in \{\b,\r\} $ for each $ j $ and $\r$ appears $n$ times while $\b$ appears $k$ times.
The symmetric group $ S_{n+k} $ acts on $ \textrm{Seq} $ with the simple transposition $s_j$ exchanging the entries in positions $j$ and $j+1$.

Let $ W_{}(n,k) $ be the algebra over $ \Bbbk $ generated by $ y_j $ for $ k=1, \ldots, n+k$,
$ \psi_j $ for $ j=1, \ldots, n+k-1$, and $ e(\bf i) $ where $ {\bf i} \in \textrm{Seq}$, satisfying the relations below.
\begin{enumerate}
\label{algebrarelations}
\item $ e({\bf i}) e({\bf j}) = \delta_{{\bf i}, {\bf j}} e({\bf i}), $
\item $ y_j e({\bf i}) = e({\bf i}) y_j, $
\item $ \psi_k e({\bf i}) = e(s_k{\bf i}) \psi_k, $
\item $ \psi_j e({\bf i}) = 0 $ if $ {\bf i}_j ={\bf i}_{j+1}=\r, $
\item $ \psi_j \psi_l e({\bf i})=\psi_l \psi_j e({\bf i}) $ if $ |j-l|>1, $
\item $ (\psi_j e({\bf i}))(y_l e({\bf i})) = (y_l e(s_j {\bf i})) (\psi_j e({\bf i})) $ if $ |j-l|>1,$
\item $ y_j y_l = y_l y_j, $
\item $ \psi_j y_j e({\bf i}) - y_{j+1} \psi_j e({\bf i}) = \delta_{{\bf i}_{j}, {\bf i}_{j+1}},$ 
\item $ y_j \psi_j e({\bf i}) -  \psi_j y_{j+1} e({\bf i}) = \delta_{{\bf i}_{j}, {\bf i}_{j+1}}, $ 
\item $ \psi_j^2 e({\bf i}) = (1-\delta_{{\bf i}_j, {\bf i}_{j+1}})   ( (-1)^{\delta_{{\bf i}_j,\r}}   y_j +  (-1)^{\delta_{{\bf i}_{j+1},\r}}    y_{j+1}) e({\bf i}), $ 
\item $ (\psi_j \psi_{j+1} \psi_j - \psi_{j+1} \psi_j \psi_{j+1}) e({\bf i}) =(-\delta_{{\bf i}_{j}, \b} \delta_{{\bf i}_{j+1}, \r} \delta_{{\bf i}_{j+2}, \b})e({\bf i}), $
\item \label{cyclotomiccondition} $\delta_{{\bf i}_{1}, \b}   e({\bf i}) =0$. 
\end{enumerate}

Let $\widetilde{W}(n,k)$ be the algebra satisfying all the relations above except the last one.
We say that ${W}(n,k)$ is the cyclotomic quotient of $\widetilde{W}(n,k)$.

This is a $\Z$-graded algebra with the degrees of generators:
\begin{align*}
deg(y_i)&=2,\\
deg(e({\bf i}))&=0, \\
deg(\psi_k e({\bf i})) &=
\begin{cases}
-2 & \text{ if } {\bf i}_k={\bf i}_{k+1} = \b, \\
1 & \text{ if } {\bf i}_k=\b \text{ and } {\bf i}_{k+1} = \r, \\
1 & \text{ if } {\bf i}_k=\r \text{ and } {\bf i}_{k+1} = \b.
\end{cases}
\end{align*}

The space $e({\bf i}) W(n,k) e({\bf j})$ has the structure of an $R_n$-module.
Suppose $\r$ occurs in the sequence ${\bf i}$ in positions $r_1, \ldots, r_n$ with $ r_1 < \cdots < r_n$.  For $ m \in e({\bf i}) W(n,k) e({\bf j})$ define
\begin{equation}
\label{Rmodulestructure}
x_l . m = y_{r_l} m, \hspace{.5in} \text{ for } l=1,\ldots,n.
\end{equation}

\subsection{Graphical presentation of $ W(n,k)$}
\label{graphicalpresentation}
There is a graphical presentation of $ W(n,k) $ very similar to the algebras introduced in ~\cite{KL1, Ro1} and ~\cite{W1}.  
We consider collections of smooth arcs in the plane connecting $ n $ red points and $k$ black points on one horizontal line with $n$ red points and $k$ black points on another horizontal line.
The arcs are colored in a manner consistent with their boundary points, and the black arcs are dashed.
Arcs are assumed to have no critical points (in other words no cups or caps).
Arcs are allowed to intersect (as long as they are both not solid red), but no triple intersections are allowed.
Arcs can carry dots.  
Two diagrams that are related by an isotopy that does not change the combinatorial types of the diagrams or the relative position of crossings are taken to be equal. 
The elements of the vector space $ W(n,k) $ are formal linear combinations of these diagrams modulo the local relations given below.
We give $ W(n,k) $ the structure of an algebra by concatenating diagrams vertically as long as the colors of the endpoints match.  If they do not, the product of two diagrams is taken to be zero.

Dots on strands correspond to generators $y_l$ given earlier.  Strands which cross correspond to generators $\psi_l$.  The action $x_l .m$ described above should be interpreted as placing a dot on the $l$-th red strand of a diagram counting from left to right (ignoring the black strands).

In ~\eqref{farawaycommute}, the rectangles labeled by $X$ and $Y$ represent the generating diagrams given in ~\eqref{degrees}.  This relation then means that far away generators commute.
The equation ~\eqref{cyclotomic} is the cyclotomic condition.

\begin{equation}
\label{farawaycommute}
\begin{tikzpicture}
\draw (.5,.25) -- (.5,0)[thick];
\draw (.5,.75) -- (.5,2)[thick];

\draw (0,.25) -- (0,.75)[thick];
\draw (1,.25) -- (1,.75)[thick];
\draw (0,.25) -- (1,.25)[thick];
\draw (0,.75) -- (1,.75)[thick];

\draw (1.5,1) node{$\cdots$};

\draw [shift={+(2,1)}](0,.25) -- (0,.75)[thick];
\draw [shift={+(2,1)}](1,.25) -- (1,.75)[thick];
\draw [shift={+(2,1)}](0,.25) -- (1,.25)[thick];
\draw [shift={+(2,1)}](0,.75) -- (1,.75)[thick];

\draw (2.5,1.25) -- (2.5,0)[thick];
\draw (2.5,1.75) -- (2.5,2)[thick];

\draw (7.5,.25) -- (7.5,0)[thick];
\draw (7.5,.75) -- (7.5,2)[thick];

\draw (.5,.5) node{$X$};
\draw (2.5,1.5) node{$Y$};

\draw (4,1) node{$=$};

\draw [shift={+(5,0)}][shift={+(0,1)}](0,.25) -- (0,.75)[thick];
\draw [shift={+(5,0)}][shift={+(0,1)}](1,.25) -- (1,.75)[thick];
\draw [shift={+(5,0)}][shift={+(0,1)}](0,.25) -- (1,.25)[thick];
\draw [shift={+(5,0)}][shift={+(0,1)}](0,.75) -- (1,.75)[thick];

\draw (5.5,1.25) -- (5.5,0)[thick];
\draw (5.5,1.75) -- (5.5,2)[thick];

\draw [shift={+(5,0)}](1.5,1) node{$\cdots$};

\draw [shift={+(5,0)}][shift={+(2,0)}](0,.25) -- (0,.75)[thick];
\draw [shift={+(5,0)}][shift={+(2,0)}](1,.25) -- (1,.75)[thick];
\draw [shift={+(5,0)}][shift={+(2,0)}](0,.25) -- (1,.25)[thick];
\draw [shift={+(5,0)}][shift={+(2,0)}](0,.75) -- (1,.75)[thick];

\draw (5.5,1.5) node{$X$};
\draw (7.5,.5) node{$Y$};

\end{tikzpicture}
\end{equation}

\begin{equation}
\label{crossdot1}
\begin{tikzpicture}[>=stealth]
\draw (0,0) -- (1,1)[][thick][dashed];
\draw (1,0) -- (0,1)[red][thick];
\filldraw [black] (.25,.25) circle (2pt);
\draw (1.5,.5) node{=};
\draw (2,0) -- (3,1)[][thick][dashed];
\draw (3,0) -- (2,1)[red][thick];
\filldraw [black] (2.75, .75) circle (2pt);

\draw [shift={(6,0)}](0,0) -- (1,1)[][thick][dashed];
\draw [shift={(6,0)}](1,0) -- (0,1)[red][thick];
\filldraw [red] [shift={(6,0)}](.25,.75) circle (2pt);
\draw [shift={(6,0)}](1.5,.5) node{=};
\draw [shift={(6,0)}](2,0) -- (3,1)[][thick][dashed];
\draw [shift={(6,0)}](3,0) -- (2,1)[red][thick];
\filldraw [red] [shift={(6,0)}](2.75, .25) circle (2pt);

\end{tikzpicture}
\end{equation}

\begin{equation}
\label{crossdot2}
\begin{tikzpicture}[>=stealth]
\draw (0,0) -- (1,1)[red][thick];
\draw (1,0) -- (0,1)[][thick][dashed];
\draw (1.5,.5) node{=};
\draw (2,0) -- (3,1)[red][thick];
\draw (3,0) -- (2,1)[][thick][dashed];
\filldraw[black](.25,.75) circle (2pt);
\filldraw[black](2.75,.25) circle (2pt);

\draw [shift={(6,0)}](0,0) -- (1,1)[red][thick];
\draw [shift={(6,0)}](1,0) -- (0,1)[][thick][dashed];
\draw [shift={(6,0)}](1.5,.5) node{=};
\draw [shift={(6,0)}](2,0) -- (3,1)[red][thick];
\draw [shift={(6,0)}](3,0) -- (2,1)[][thick][dashed];
\filldraw[red] [shift={(6,0)}](.25,.25) circle (2pt);
\filldraw[red] [shift={(6,0)}](2.75,.75) circle (2pt);

\end{tikzpicture}
\end{equation}

\begin{equation}
\label{mm+1}
\begin{tikzpicture}
\draw (0,0) .. controls (1,1) .. (0,2)[red][thick];
\draw (1,0) .. controls (0,1) .. (1,2)[][thick][dashed];
\draw (1.5,1) node{$=$};
\draw (2,0) -- (2,2)[red][thick];
\draw (3,0) -- (3,2)[thick][dashed];
\filldraw[black](3,1) circle (2pt);
\draw (3.5,1) node{$-$};
\draw (4,0) -- (4,2)[red][thick];
\draw (5,0) -- (5,2)[thick][dashed];
\filldraw[red](4,1) circle (2pt);
\end{tikzpicture}
\end{equation}

\begin{equation}
\begin{tikzpicture}
\draw (5,0) .. controls (6,1) .. (5,2)[][thick][dashed];
\draw (6,0) .. controls (5,1) .. (6,2)[red][thick];
\draw (6.5,1) node{$=$};
\draw (7,0) -- (7,2)[black][thick][dashed];
\draw (8,0) -- (8,2)[red][thick];
\filldraw[black](7,1) circle (2pt);
\draw (8.5,1) node{$-$};
\draw (9,0) -- (9,2)[black][thick][dashed];
\draw (10,0) -- (10,2)[red][thick];
\filldraw[red](10,1) circle (2pt);
\end{tikzpicture}
\end{equation}

\begin{equation}
\label{nilhecke3}
\begin{tikzpicture}
\draw (0,0) .. controls (1,1) .. (0,2)[black][thick][dashed];
\draw (1,0) .. controls (0,1) .. (1,2)[black][thick][dashed];
\draw (1.5,1) node{$=$};
\draw (2,1) node{$0$};
\end{tikzpicture}
\end{equation}

\begin{equation}
\label{nilhecke2}
\begin{tikzpicture}[>=stealth]
\draw (0,0) -- (1,1)[][thick][dashed];
\draw (1,0) -- (0,1)[][thick][dashed];
\filldraw [black] (.25,.25) circle (2pt);
\draw (1.5,.5) node{-};
\draw (2,0) -- (3,1)[][thick][dashed];
\draw (3,0) -- (2,1)[][thick][dashed];
\filldraw [black] (2.75, .75) circle (2pt);
\draw (3.5,.5) node{=};
\draw (4,0) -- (4,1)[][thick][dashed];
\draw (5,0) -- (5,1)[][thick][dashed];
\draw (5.5,.5) node{=};

\draw [shift={(6,0)}](0,0) -- (1,1)[][thick][dashed];
\draw [shift={(6,0)}](1,0) -- (0,1)[][thick][dashed];
\filldraw [black] [shift={(6,0)}](.25,.75) circle (2pt);
\draw [shift={(6,0)}](1.5,.5) node{-};
\draw [shift={(6,0)}](2,0) -- (3,1)[][thick][dashed];
\draw [shift={(6,0)}](3,0) -- (2,1)[][thick][dashed];
\filldraw [black] [shift={(6,0)}](2.75, .25) circle (2pt);

\end{tikzpicture}
\end{equation}

\begin{equation}
\label{reid3}
\begin{tikzpicture}[>=stealth]
\draw (0,0) -- (2,2)[black][thick][dashed];
\draw (2,0) -- (0,2)[black][thick][dashed];
\draw (1,0) .. controls (0,1) .. (1,2)[black][thick][dashed];
\draw (2.5,1) node {=};
\draw (3,0) -- (5,2)[black][thick][dashed];
\draw (5,0) -- (3,2)[black][thick][dashed];
\draw (4,0) .. controls (5,1) .. (4,2)[black][thick][dashed];
\end{tikzpicture}
\end{equation}

\begin{equation}
\begin{tikzpicture}[>=stealth]
\draw (0,0) -- (2,2)[black][thick][dashed];
\draw (2,0) -- (0,2)[red][thick];
\draw (1,0) .. controls (0,1) .. (1,2)[black][thick][dashed];
\draw (2.5,1) node {=};
\draw (3,0) -- (5,2)[black][thick][dashed];
\draw (5,0) -- (3,2)[red][thick];
\draw (4,0) .. controls (5,1) .. (4,2)[black][thick][dashed];

\draw [shift={+(7,0)}](0,0) -- (2,2)[red][thick];
\draw [shift={+(7,0)}](2,0) -- (0,2)[black][thick][dashed];
\draw [shift={+(7,0)}](1,0) .. controls (0,1) .. (1,2)[black][thick][dashed];
\draw [shift={+(7,0)}](2.5,1) node {=};
\draw [shift={+(7,0)}](3,0) -- (5,2)[red][thick];
\draw [shift={+(7,0)}](5,0) -- (3,2)[black][thick][dashed];
\draw [shift={+(7,0)}](4,0) .. controls (5,1) .. (4,2)[black][thick][dashed];

\end{tikzpicture}
\end{equation}

\begin{equation}
\label{reid3nu}
\begin{tikzpicture}[>=stealth]
\draw (0,0) -- (2,2)[black][thick][dashed];
\draw (2,0) -- (0,2)[black][thick][dashed];
\draw (1,0) .. controls (0,1) .. (1,2)[red][thick];
\draw (2.5,1) node {$-$};
\draw (3,0) -- (5,2)[black][thick][dashed];
\draw (5,0) -- (3,2)[black][thick][dashed];
\draw (4,0) .. controls (5,1) .. (4,2)[red][thick];
\draw (6.25,1) node{= };
\draw (7.05,1) node{$-$ };
\draw (7.5,0) -- (7.5,2)[black][thick][dashed];
\draw (8.5,0) -- (8.5,2)[red][thick];
\draw (9.5,0) -- (9.5,2)[black][thick][dashed];
\end{tikzpicture}
\end{equation}

\begin{equation}
\label{cyclotomic}
\begin{tikzpicture}[>=stealth]
\draw (0,0) -- (0,2)[black][thick][dashed];

\draw (1,1) node{$\cdots$ };
\draw (2,1) node{$= 0$ };
\end{tikzpicture}
\end{equation}

\begin{equation}
\label{degrees}
\begin{tikzpicture}
\draw (-1,0) -- (14,0)[][very thick];
\draw (0,.5) node{degree};
\draw (0, -1) node{generator};
\draw (1,1) -- (1,-2)[][very thick];
\draw (1.5,.5) node{$0$};
\draw (1.5,-.5) -- (1.5,-1.5)[red][thick];

\draw (2,1) -- (2, -2)[very thick];
\draw (4.5,-.5) -- (4.5,-1.5)[thick][dashed];
\draw (4,1) -- (4, -2)[very thick];
\draw (5,1) -- (5, -2)[very thick];
\draw (4.5, .5) node{$2$};
\filldraw[black](4.5,-1) circle (2pt);
\draw (6,-1.5) -- (7,-.5)[black][thick][dashed];
\draw (7, -1.5) -- (6, -.5)[black][thick][dashed];

\draw (6.5,.5) node{$-2$};
\draw (-1,1) -- (-1,-2)[very thick];
\draw (8,1) -- (8,-2)[very thick];
\draw (-1,1) -- (14,1)[][very thick];
\draw (-1,-2) -- (14,-2)[][very thick];
\draw (11,1) -- (11,-2)[very thick];
\draw (9,-1.5) -- (10,-.5)[black][thick][dashed];
\draw (10, -1.5) -- (9, -.5)[red][thick];
\draw (9.5,.5) node{$1$};
\draw (11,1) -- (11,-2)[very thick];
\draw (12,-1.5) -- (13,-.5)[red][thick];
\draw (13, -1.5) -- (12, -.5)[black][thick][dashed];
\draw (12.5,.5) node{$1$};
\draw (14,1) -- (14,-2)[very thick];
\draw (3,1) -- (3,-2)[black][thick];
\draw (2.5,.5) node{$0$};
\draw (2.5,-.5) -- (2.5,-1.5)[black][thick][dashed];
\draw (3.5,.5) node{$2$};
\draw (3.5,-.5) -- (3.5,-1.5)[red][thick];
\filldraw[red](3.5,-1) circle (2pt);

\end{tikzpicture}
\end{equation}

\begin{remark}
There is a surjective map $ {W}(n,k) \rightarrow T^n_{n-2k} $, where $T^n_{n-2k}$ is the algebra defined by Webster in ~\cite[Definition 2.2]{W2},
by setting all diagrams with a dot on a solid red strand to zero in the graphical presentation of ${W}(n,k)$.
We refer to $W(n,k)$ as a redotted Webster algebra for $\mathfrak{sl}_2$.

This implies that the algebra $W(n,k)$ is non-zero for $0 \leq k \leq n$.  In Section ~\ref{sectionW(n,1)} we prove directly that $W(n,1)$ is non-trivial.
\end{remark}

\subsection{The algebra $W^p(n,k)$}
\label{subsectionW^p(n,k)}
For $p=1,\ldots,n-1$ we will define an algebra $W^p(n,k)$ isomorphic to a subalgebra of $W(n,k)$ whose inclusion is non-unital.

Let $ \textrm{Seq}_p $ be the set of all sequences $ {\bf i}=(i_1, \ldots, i_{n+k-1}) $ where $ i_{j} \in \{\b,\r,\R\} $ for each $ j $.  We assume $\r$ appears $n-2$ times while $\b$ appears $k$ times and $\R$ appears just once.  We also assume that the symbol $\r$ appears exactly $p-1$ times before the symbol $\R$ appears.
The symmetric group $ S_{n+k-1} $ acts on $ \textrm{Seq}_p $ with the simple transposition $s_j$ exchanging the entries in positions $j$ and $j+1$.

Let $ W^{p}(n,k) $ be the algebra over $ \Bbbk $ generated by $ y_j $ for $ k=1, \ldots, n+k-1$, along with $E_1$, $E_2$,
$ \psi_j $ for $ j=1, \ldots, n+k-2$, and $ e(\bf i) $ where $ {\bf i} \in \textrm{Seq}_p$, satisfying the relations below.

Informally, $\R$ represents two solid red strands merged into one. $E_1$ and $ E_2$ represent the first and second elementary symmetric functions in the dots on this merged strand.

\begin{enumerate}
\label{algebrarelations2}
\item $ e({\bf i}) e({\bf j}) = \delta_{{\bf i}, {\bf j}} e({\bf i}), $
\item $ y_j e({\bf i}) = e({\bf i}) y_j, $
\item $y_j e({\bf i})=0$ if ${\bf i}_j=\R,$
\item $E_1$ and $E_2$ are central,
\item $ \psi_k e({\bf i}) = e(s_k{\bf i}) \psi_k, $
\item $ \psi_j e({\bf i}) = 0 $ if $ {\bf i}_j , {\bf i}_{j+1} \in \{\r, \R \}, $
\item $ \psi_j \psi_l e({\bf i})=\psi_l \psi_j e({\bf i}) $ if $ |j-l|>1 $
\item $ (\psi_j e({\bf i}))(y_l e({\bf i})) = (y_l e(s_j {\bf i})) (\psi_j e({\bf i})) $ if $ |j-l|>1,$
\item $ y_j y_l = y_l y_j, $
\item $ \psi_j y_j e({\bf i}) - y_{j+1} \psi_j e({\bf i}) = \delta_{{\bf i}_{j}, {\bf i}_{j+1}},$ 
\item $ y_j \psi_j e({\bf i}) -  \psi_j y_{j+1} e({\bf i}) = \delta_{{\bf i}_{j}, {\bf i}_{j+1}}, $ 
\item \begin{equation*}
\psi_j^2 e({\bf i})=
\begin{cases}
0 & \text{ if }  {\bf i}_j = {\bf i}_{j+1}=\b\\
(-y_j+y_{j+1})e({\bf i}) & \text{ if } {\bf i}_j=\r, {\bf i}_{j+1}=\b \\
(y_j-y_{j+1})e({\bf i}) & \text{ if } {\bf i}_j=\b, {\bf i}_{j+1}=\r \\
(E_2-E_1y_{j+1} + y_{j+1}^2)e({\bf i}) & \text{ if } {\bf i}_j=\R, {\bf i}_{j+1}=\b \\
(E_2-y_{j}E_1 + y_{j}^2)e({\bf i}) & \text{ if } {\bf i}_j=\b, {\bf i}_{j+1}=\R, 
\end{cases}
\end{equation*}
\item $ (\psi_j \psi_{j+1} \psi_j - \psi_{j+1} \psi_j \psi_{j+1}) e({\bf i}) =(-\delta_{{\bf i}_{j}, \b} \delta_{{\bf i}_{j+1}, \r} \delta_{{\bf i}_{j+2}, \b} 
+ (-y_i +E_1 - y_{i+2})   \delta_{{\bf i}_{j}, \b} \delta_{{\bf i}_{j+1}, \R} \delta_{{\bf i}_{j+2}, \b})e({\bf i}), $
\item \label{cyclotomiccondition} $\delta_{{\bf i}_{1}, \b}   e({\bf i})=0 $. 
\end{enumerate}

This is a graded algebra with the degrees of generators:
\begin{align*}
deg(y_i)&=2, \hspace{.3in} deg(E_1)=2, \hspace{.3in} deg(E_2)=4, \hspace{.3in} deg(e({\bf i}))=0,  \\
deg(\psi_k e({\bf i})) &=
\begin{cases}
-2 & \text{ if } {\bf i}_k={\bf i}_{k+1} = \b, \\
1 & \text{ if } {\bf i}_k=\b \text{ and } {\bf i}_{k+1} = \r, \\
1 & \text{ if } {\bf i}_k=\r \text{ and } {\bf i}_{k+1} = \b, \\
2 & \text{ if } {\bf i}_k=\R \text{ and } {\bf i}_{k+1} = \b, \\
2 & \text{ if } {\bf i}_k=\b \text{ and } {\bf i}_{k+1} = \R.
\end{cases}
\end{align*}

There is a graphical description of $W^p(n,k)$ but in the interest of space we only sketch it for $W^p(n,1)$ in Section ~\ref{specialsubalgebra}.

\begin{remark}
There is an injective homomorphism $\rho_p \colon W^p(n,k) \rightarrow W(n,k)$.  We will prove this later for the special case that $k=1$.
\end{remark}

\begin{remark}
We define $~\widetilde{W}^p(n,k)$ to be the algebra defined above with all of the relations except the cyclotomic relation $\delta_{{\bf i}_{1}, \b}   e({\bf i})=0 $. However we will not make use of this algebra in this paper.
\end{remark}

\section{Properties of ${W}(n,1)$}
\label{sectionW(n,1)}
In this section and throughout the rest of the paper we will specialize the algebras introduced in Section ~\ref{sectionW(n,k)} to the case $k=1$.  

It will be convenient to write  $e_i=e({\bf i}) $ where ${\bf i}$ is the sequence where the only $\b$ occurs in position $i$.  By the last defining relation (the cyclotomic relation) in Section ~\ref{algebrarelationsI}, it follows that $e_1=0$.

\subsection{A representation of ${W}(n,1)$}
\label{repofW(n,1)}
Recall that $R=\Bbbk[x_1, \ldots, x_n]$.  Define the vector space 
$V_{n,i+1}=R[y_i]/((y_i-x_1) \cdots (y_i-x_i)) $ and $ V_n=\oplus_{i=2}^{n+1} V_{n,i}$.

\begin{itemize}

\item We define $e_i$ to act trivially on $V_{n,j}$ unless $i=j$ in which case $e_i$ acts as the identity on $V_{n,i}$.  A dot on the $l$-th red strand of a diagram (with no crossings) acts as multiplication by $x_l$.
A dot on the black strand acts as multiplication by $y_i$.

\item A crossing taking the element $\b$ in the $i$-th position on the bottom to the $(i-1)$-st position on the top maps $1 \in V_{n,i}$ to $ 1 \in V_{n,i-1}$.

\item A crossing taking the element $\b$ in the $i$-th position on the bottom to the $(i+1)$-st position on the top maps $1 \in V_{n,i}$ to $ y_{i}-x_{i} \in V_{n,i+1}$.

\end{itemize}

\begin{prop}
$V_n$ is a representation of $W(n,1)$.
\end{prop}

\begin{proof}
This is a straightforward check.
\end{proof}

\subsection{A basis}
The next result describes a basis of $e_j W(n,1) e_i$.  One draws a diagram with minimal number of crossings from the sequence corresponding to $e_i$ on the bottom to the sequence corresponding to $e_j$ on the top and then decorate on the top freely with dots on the solid red strands but restrictively (depending on how close the dashed black strand gets to the left) with dots on the dashed (black) strand.

\begin{prop}
\label{basisprop}

A basis of $ e_j {W}(n,1) e_i$ over $\Bbbk$ for $i,j \geq 2$ is given by
\begin{equation}
\label{basisforW(n,1)}
\{e_j y_1^{a_1} \cdots y_j^{c_j} \cdots y_{n+1}^{a_{n+1}}  \psi_w e_i | a_i \in \Z_{\geq 0}, 0 \leq c_j \leq min(i,j)-2    \},
\end{equation}

where
$w \in S_{n+1} $ is of minimal length taking the sequence 
$(\underbrace{\r \ldots \r}_{i-1} \b \underbrace{\r \ldots \r}_{n-i+1}) $ to 
$(\underbrace{\r \ldots \r}_{j-1} \b \underbrace{\r \ldots \r}_{n-j+1}) $.

If $i=1$ or $j=1$, $ e_j {W}(n,1) e_i=0$ by the cyclotomic condition.
\end{prop}

\begin{proof}
We begin by showing that the proposed basis actually spans.

First assume that $i=j$.
We claim that

\begin{equation}
\label{spanii}
\begin{tikzpicture}
\draw (0,0) -- (0,2)[red][thick];
\draw (1,0) -- (1,2)[red][thick];
\draw (.5,1) node{$\cdots$};
\draw (2,0) -- (2,2)[black][thick][dashed];
\draw (3.5,1) node{$\cdots$};
\draw (3,0) -- (3,2)[red][thick];
\draw (4,0) -- (4,2)[red][thick];

\filldraw [red] (0, 1) circle (2pt);
\filldraw [red] (1, 1) circle (2pt);
\filldraw [black] (2, 1) circle (2pt);
\filldraw [red] (3, 1) circle (2pt);
\filldraw [red] (4, 1) circle (2pt);

\draw (-.25,1.25) node{$a_1$};
\draw (.65,1.25) node{$a_{i-1}$};
\draw (1.75,1.25) node{$c_i$};
\draw (2.65,1.25) node{$a_{i+1}$};
\draw (3.65,1.25) node{$a_{n+1}$};

\end{tikzpicture}
\end{equation}
for $ a_i \in \Z_{\geq 0}$ and $0 \leq c_i \leq i-2  $ span $ e_i {W}(n,1) e_i$.

In order to see this consider the diagram in ~\eqref{spaniia} which is zero by the cyclotomic condition.
\begin{equation}
\label{spaniia}
\begin{tikzpicture}
\draw (2,0) .. controls (-1.5,1) .. (2,2)[][thick][dashed];
\draw (0,0) -- (0,2)[red][thick];
\draw (1,0) -- (1,2)[red][thick];
\draw (.5,1) node{$\cdots$};
\draw (3.5,1) node{$\cdots$};
\draw (3,0) -- (3,2)[red][thick];
\draw (4,0) -- (4,2)[red][thick];

\end{tikzpicture}
\end{equation}
Using the relations to remove $i-1$ double crossings in ~\eqref{spaniia} along with sliding dots through crossings 
we get the relation
\begin{equation*}
y_i^{i-1} -(y_1+\cdots+y_{i-1})y_i^{i-2} + \cdots + (-1)^{i-1} (y_1 \cdots y_{i-1})=0.
\end{equation*}
Thus in ~\eqref{spanii} we may take $c_i \leq i-2$.
It then follows with little effort that the set in ~\eqref{basisforW(n,1)} actually spans.

Now we must prove that the elements in ~\eqref{basisforW(n,1)} are linearly independent.

When $i=j$ the elements in
\begin{equation*}
\{e_i y_1^{a_1} \cdots y_i^{c_i} \cdots y_{n+1}^{a_{n+1}}  e_i | a_i \in \Z_{\geq 0}, 0 \leq c_i \leq i-2    \}.
\end{equation*}
are linearly independent.  Otherwise as a linear operator on $V_{n,i}$ the element $y_i$ satisfies a polynomial with coefficients in $R$ of degree less than $i-1$.

Linear independence for the general case of elements in ~\eqref{basisforW(n,1)} now follows easily.
\end{proof}

\begin{corollary}
The representation of $W(n,1)$ on $V_n$ is faithful.
\end{corollary}

\subsection{The center}

Recall that $W(n,1)$ (more generally, $W(n,k)$) has the structure of an $R$-module given in ~\eqref{Rmodulestructure}.  The action was denoted by a dot.

\begin{lemma}
\label{centerlemma1}
For $k=2, \ldots, n+1$, there is an equality in $W(n,1)$
\begin{equation*}
y_k^n e_k - E_1(x_1, \ldots, x_n).y_k^{n-1} e_k  + E_2(x_1, \ldots, x_n).y_k^{n-2} e_k + \cdots + (-1)^{n} E_n(x_1, \ldots, x_n).y_k^0 e_k = 0.
\end{equation*}
\end{lemma}

\begin{proof}
By the relations in $W(n,1)$ we may resolve $n-k+1$ double crossings and get
\begin{equation*}
y_k^{k-1} (\psi_k \cdots \psi_n)  (\psi_n \cdots  \psi_k)  e_k
=
\sum_{j=0}^{n-k+1} (-1)^j  E_j(x_k, \ldots, x_n) .( y_k^{k-1} y_k^{n-k+1-j} e_k).
\end{equation*}

Thus
\begin{equation}
\label{centerlemmaeq1}
y_k^n e_k = y_k^{k-1} (\psi_k \cdots \psi_n)  (\psi_n \cdots  \psi_k)  e_k
+
\sum_{j=1}^{n-k+1} (-1)^{j+1}  E_j(x_k, \ldots, x_n) .( y_k^{n-j}  e_k).
\end{equation}
Using the relations in $W(n,1)$ for $y_k^{k-1} e_k$ in the first term in the right hand side of equation ~\eqref{centerlemmaeq1} we get
\begin{align}
\label{centerlemmaeq2}
y_k^n e_k = &\sum_{j=1}^{k-1} (-1)^{j+1} E_j(x_1, \ldots, x_{k-1}) . 
y_k^{k-j-1} (\psi_k \cdots \psi_n)  (\psi_n \cdots  \psi_k)  e_k
+ \\
& \sum_{j=1}^{n-k+1} (-1)^{j+1}  E_j(x_k, \ldots, x_n) .( y_k^{n-j}  e_k). \nonumber
\end{align}
Using relations in $W(n,1)$ to remove double crossings in the first sum of equation ~\eqref{centerlemmaeq2} we get
\begin{align}
\label{centerlemmaeq3}
y_k^n e_k = & \sum_{j=1}^{k-1}  \sum_{r=0}^{n-k+1} (-1)^{j+r+1}  E_j(x_1, \ldots, x_{k-1}) E_r(x_k, \ldots, x_n) . y_k^{n-j-r}  e_k + \\
& \sum_{j=1}^{n-k+1} (-1)^{j+1}  E_j(x_k, \ldots, x_n) .( y_k^{n-j}  e_k). \nonumber
\end{align}
Grouping together terms according to the exponent of $y_k$ in equation ~\eqref{centerlemmaeq3} we arrive at
\begin{align*}
y_k^n e_k &= \sum_{p=0}^{n-1} (-1)^{n-1+p} \sum_{\gamma} E_{\gamma}(x_1, \ldots, x_{k-1}) E_{n-p-\gamma}(x_k, \ldots, x_n) . y_k^p e_k \\
&= \sum_{p=0}^{n-1} (-1)^{n-1+p} E_{n-p}(x_1, \ldots, x_n). y_k^p e_k
\end{align*}
which proves the lemma.
\end{proof}

\begin{prop}
\label{centerofW(n,1)}
The center of $W(n,1)$ is isomorphic as an algebra to 
$R[z]/(z-x_1)\cdots(z-x_n)$.
\end{prop}

\begin{proof}
We first check that the center is generated as an $R$-module by the set
\begin{equation*}
\{z^k = \sum_{i=2}^{n+1} y_i^k e_i | k=0,\ldots,n-1 \}.
\end{equation*}
Lemma ~\ref{centerlemma1} shows that $z$ satisfies the appropriate relation.

By applying idempotents $e_i$ on the left and right of a central element $z$ it is clear that we may write $z$ as:
\begin{equation*}
z=\sum_{i=2}^{n+1} e_i f_i(y_1,\ldots,y_{n+1}) e_i = \sum_{i=2}^{n+1}  f_i(y_1,\ldots,y_{n+1}) e_i
\end{equation*}
where $f_i$ is a polynomial in the $y_j$.

Rewriting $z$ in the basis of $W(n,1)$ described in ~\eqref{basisforW(n,1)} we have
\begin{equation*}
z= \sum_{i=2}^{n+1} \sum_{j=0}^{i-2} g_{ij}(x_1, \ldots, x_{n} ).y_i^j e_i.
\end{equation*}
We will abbreviate $g_{ij}=g_{ij}(x_1, \ldots, x_{n} )$. 

We would like to show that
\begin{equation*}
z= \sum_{i=2}^{n+1} \sum_{j=0}^{n-1} g_{n+1,j}. y_i^j e_i
\end{equation*}
or equivalently that
\begin{equation}
\label{goalforformofz}
ze_k= \sum_{j=0}^{n-1} g_{n+1,j}. y_k^j e_k.
\end{equation}

Since $z$ is central,
$ (e_{k+1} \psi_k e_k) z = z (e_{k+1} \psi_k e_k)$.
We compute
\begin{equation*}
 (e_{k+1} \psi_k e_k) z =
 \sum_{j=0}^{k-2} g_{kj}. \psi_k y_k^j e_k.
\end{equation*}

On the other hand
\begin{align*}
z(e_{k+1} \psi_k e_k) &= \sum_{j=0}^{k-1} g_{k+1,j} .y_{k+1}^j e_{k+1} \psi_k e_k \\
&= \sum_{j=0}^{k-1} g_{k+1,j} .\psi_k y_k^j e_k \\
&= \sum_{j=0}^{k-2} (g_{k+1,j}+ (-1)^{k-j} E_{k-1-j}(x_1, \ldots, x_{k-1}) g_{k+1,k-1} ) .\psi_k y_k^j e_k
\end{align*}
where we used relations in $W(n,1)$ to establish the third equality above.

Now that we have expressed $(e_{k+1} \psi_k e_k) z$ and $ z (e_{k+1} \psi_k e_k)$
in terms of the basis described earlier we conclude
\begin{equation}
\label{grecursion}
g_{k,j} = g_{k+1,j} + (-1)^{k-j} E_{k-1-j}(x_1, \ldots, x_{k-1}) g_{k+1,k-1}.
\end{equation}

Thus
\begin{align*}
z e_k &= \sum_{j=0}^{k-2} g_{k,j}. y_k^j e_k \\
&= \sum_{j=0}^{k-2} g_{k+1,j}. y_k^j e_k + 
\sum_{j=0}^{k-2} (-1)^{k-j} g_{k+1,k-1} E_{k-1-j}(x_1, \ldots, x_{k-1}). y_k^j e_k \\
&= \sum_{j=0}^{k-1} g_{k+1,j}. y_k^j e_k.
\end{align*}

Equality ~\eqref{goalforformofz} follows by iterating the recursion ~\eqref{grecursion}. 
Therefore, as an $R$-module the center is generated by $z^0, \ldots, z^{n-1}$.
The fact that $z$ satisfies the relation in the Proposition follows from Lemma ~\ref{centerlemma1}.

\end{proof}

\begin{ex}
For $n=2$, the generating set of the center as a free $R$-module is graphically depicted below. 

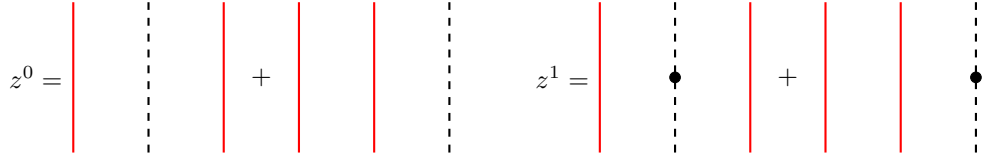
\begin{equation*}
\begin{tikzpicture}
\draw (.5,1) node{$z^0=$};

\draw (1,0) -- (1,2)[red][thick];
\draw (2,0) -- (2,2)[black][thick][dashed];
\draw (3,0) -- (3,2)[red][thick];
\draw (3.5,1) node{$+$};

\draw (4,0) -- (4,2)[red][thick];
\draw (6,0) -- (6,2)[black][thick][dashed];
\draw (5,0) -- (5,2)[red][thick];

\draw (7.5,1) node{$z^1=$};

\draw (8,0) -- (8,2)[red][thick];
\draw (9,0) -- (9,2)[black][thick][dashed];
\draw (10,0) -- (10,2)[red][thick];
\draw (10.5,1) node{$+$};

\draw (11,0) -- (11,2)[red][thick];
\draw (13,0) -- (13,2)[black][thick][dashed];
\draw (12,0) -- (12,2)[red][thick];

\filldraw[black](9,1) circle (2pt);
\filldraw[black](13,1) circle (2pt);

\end{tikzpicture}
\end{equation*}

\end{ex}

\subsection{A subalgebra} 
\label{specialsubalgebra}
Set $W^p := W^p(n,1)$, where $W^p(n,k)$ was defined in Section ~\ref{subsectionW^p(n,k)}.
Let $e_{i,p}$ be the idempotent corresponding to the sequence where the single entry $\b$ appears in position $i$.

We may describe $W^p(n,1)$ in a graphical way just as we did for $W(n,1)$ in section ~\ref{graphicalpresentation}.
Along with the generators for $W(n,1)$ graphically described earlier, we also have a thick red line which can carry dots labeled $E_1$ or $E_2$.  The thick red line may also intersect with the black dashed line.

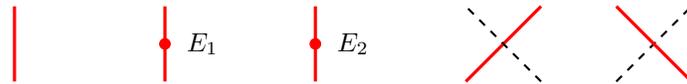
\begin{equation*}
\label{subgenerators}
\begin{tikzpicture}
\draw (0,0) -- (0,1)[very thick][red];

\draw (2,0) -- (2,1)[very thick][red];
\filldraw [red] (2, .5) circle (2pt);
\draw (2.5,.5) node{$E_1$};

\draw (4,0) -- (4,1)[very thick][red];
\filldraw [red] (4, .5) circle (2pt);
\draw (4.5,.5) node{$E_2$};

\draw (6,0) -- (7,1)[very thick][red];
\draw (7,0) -- (6,1)[thick][dashed][black];

\draw (8,0) -- (9,1)[thick][dashed][black];
\draw (9,0) -- (8,1)[very thick][][red];


\end{tikzpicture}
\end{equation*}

In addition to the local relations in $W(n,1)$ where there is only one (black) dashed and solid thin (red) strands, there are also relations
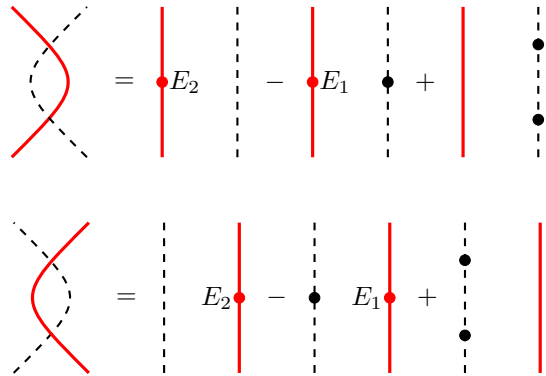
\begin{equation*}
\begin{tikzpicture}
\draw (0,0) .. controls (1,1) .. (0,2)[red][very thick];
\draw (1,0) .. controls (0,1) .. (1,2)[dashed][thick];
\draw (1.5,1) node{$=$};

\draw (2,0) -- (2,2)[very thick][red];
\draw (3,0) -- (3,2)[thick][dashed][black];
\filldraw [red] (2, 1) circle (2pt);
\draw (2.3,1) node{$E_2$};

\draw (3.5,1) node{$-$};

\draw (4,0) -- (4,2)[very thick][red];
\draw (5,0) -- (5,2)[thick][dashed][black];
\filldraw [red] (4, 1) circle (2pt);
\draw (4.3,1) node{$E_1$};
\filldraw [black] (5, 1) circle (2pt);

\draw (5.5,1) node{$+$};

\draw (6,0) -- (6,2)[very thick][red];
\draw (7,0) -- (7,2)[thick][dashed][black];
\filldraw [black] (7, .5) circle (2pt);
\filldraw [black] (7, 1.5) circle (2pt);

\end{tikzpicture}
\end{equation*}

\begin{equation*}
\begin{tikzpicture}
\draw (0,0) .. controls (1,1) .. (0,2)[dashed][thick];
\draw (1,0) .. controls (0,1) .. (1,2)[red][very thick];
\draw (1.5,1) node{$=$};

\draw (2,0) -- (2,2)[thick][dashed][black];
\draw (3,0) -- (3,2)[very thick][][red];
\filldraw [red] (3, 1) circle (2pt);
\draw (2.7,1) node{$E_2$};

\draw (3.5,1) node{$-$};

\draw (4,0) -- (4,2)[thick][dashed][black];    
\draw (5,0) -- (5,2) [very thick][red];
\filldraw [red] (5, 1) circle (2pt);
\draw (4.7,1) node{$E_1$};
\filldraw [black] (4, 1) circle (2pt);

\draw (5.5,1) node{$+$};

\draw (6,0) -- (6,2)[thick][dashed][black]; 
\draw (7,0) -- (7,2)[very thick][red];
\filldraw [black] (6, .5) circle (2pt);
\filldraw [black] (6, 1.5) circle (2pt);

\end{tikzpicture}
\end{equation*}
There are also some obvious other relations coming from relations in $W(n,1)$ by
treating the thick (red) line as two thin red lines and the dots $E_i$ as the sum or product of red dots on thin (red) lines.  For example there is an equality:
\begin{equation*}
\label{otherobviousrelationsrexample}
\begin{tikzpicture}

\filldraw [red] (6.25, .25) circle (2pt);
\draw (5.9,.3) node{$E_i$};

\filldraw [red] (8.75, .75) circle (2pt);
\draw (9.1,.7) node{$E_i$};

\draw (6,0) -- (7,1)[very thick][red];
\draw (7,0) -- (6,1)[thick][dashed][black];

\draw (7.5,.5) node{$=$};

\draw (8,0) -- (9,1) [very thick][][red];
\draw (9,0) -- (8,1) [thick][dashed][black];

\end{tikzpicture}
\end{equation*}

In order to construct a basis of $W^p(n,1)$ we will find a representation for it just as we did for $W(n,1)$.
Recall the vector space $V_n$ defined in Section ~\ref{repofW(n,1)} and consider the subspace
\begin{equation*}
V_n^p = \bigoplus_{\substack{i=2 \\ i \neq p}}^{n+1} V_{n,i}
\end{equation*}

We construct an action of $W^p(n,1)$ on $V_n^p$ as follows.
\begin{itemize}

\item For $ i \leq p$ the idempotent $e_{i,p} $ acts trivially on $V_{n,j}^p$ unless $j=i-1$ in which case $e_i$ acts as the identity on $V_{n,i-1}^p$.  
If $i >p$ then $e_{i,p}$ acts non-trivially only on $V_{n,i}^p$ in which case it acts as the identity.

\item
Assuming $l<p$, a dot on the $l$-th solid, thin, red strand of a diagram (with no crossings) acts as multiplication by $x_l$.

Assuming $l \geq p$, a dot on the $l$-th solid, thin, red strand of a diagram (with no crossings) acts as multiplication by $x_{l+2}$.

A dot labeled by $E_1$ on a thick, solid, red strand acts by $y_p + y_{p+1}$.
A dot labeled by $E_2$ on a thick, solid, red strand acts by $y_p y_{p+1}$.

A dot on the black strand in position $i$ of a diagram with no crossing acts as multiplication by $y_i$.

\item If $p \geq i$, a crossing taking the the element $\b$ in the $i$-th position on the bottom to the $(i-1)$-st position on the top maps $1 \in V_{n,i-1}$ to $ 1 \in V_{n,i-2}$.

If $p = i-1$, a crossing taking the the element $\b$ in the $i$-th position on the bottom to the $(i-1)$-st position on the top maps $1 \in V_{n,i-1}$ to $ 1 \in V_{n,i-3}$.

If $p \leq i-2$, a crossing taking the the element $\b$ in the $i$-th position on the bottom to the $(i-1)$-st position on the top maps $1 \in V_{n,i}$ to $ 1 \in V_{n,i-1}$.

\item If $i < p$, a crossing taking the the element $\b$ in the $i$-th position on the bottom to the $(i+1)$-st position on the top maps $1 \in V_{n,i-1}$ to $ y_{i-1}-x_{i-1} \in V_{n,i}$.

If $i = p$, a crossing taking the the element $\b$ in the $i$-th position on the bottom to the $(i+1)$-st position on the top maps $1 \in V_{n,i-1}$ to $ (y_{i}-x_{i-1})(y_{i}-x_{i}) \in V_{n,i+1}$.

If $i > p$, a crossing taking the the element $\b$ in the $i$-th position on the bottom to the $(i+1)$-st position on the top maps $1 \in V_{n,i}$ to $ y_{i}-x_{i+1} \in V_{n,i+1}$.

\end{itemize}

\begin{prop}
$V_n^p$ is a representation of $W^p(n,1)$.
\end{prop}

\begin{proof}
This is a straightfoward check.
\end{proof}

\begin{prop}
\label{subbasisprop}
For $i,j \geq 2$,
let $w \in S_{n} $ be of minimal length taking the sequence 
with $\b$ occurring in entry $i$ to the sequence with $\b$ occurring in entry $j$.
A basis of $ e_j {W}^p(n,1) e_i$ over $\Bbbk$ is given by
\begin{itemize}
\item $ \{e_j E_1^{a_{-1}} E_2^{a_{-2}} y_1^{a_1} \cdots y_j^{c_j} \cdots y_{n}^{a_{n}}  \psi_w e_i | a_i \in \Z_{\geq 0}, a_{p+1}=0, 0 \leq c_j \leq min(i,j)-2    \} $ 
if $j  \leq p$
\item $ \{e_j E_1^{a_{-1}} E_2^{a_{-2}} y_1^{a_1} \cdots y_j^{c_j} \cdots y_{n}^{a_{n}}  \psi_w e_i | a_i \in \Z_{\geq 0}, a_{p}=0, 0 \leq c_j \leq min(i,j)-2    \} $ 
if $j  >p, i \leq p$
\item $ \{e_j E_1^{a_{-1}} E_2^{a_{-2}}  y_1^{a_1} \cdots y_j^{c_j} \cdots y_{n}^{a_{n}}  \psi_w e_i | a_i \in \Z_{\geq 0}, a_p=0, 0 \leq c_j \leq min(i,j)-1    \} $ 
if $min(i,j) > p$.
\end{itemize}
\end{prop}

\begin{proof}
The proof of this is similar to the proof of Proposition ~\ref{basisprop}.
\end{proof}

\begin{corollary}
The representation of $W^p(n,1)$ on $V_n^p$ is faithful.
\end{corollary}

\begin{prop}
There is an injective homomorphism $\rho_p \colon W^p(n,1) \rightarrow W(n,1)$.
\end{prop}

\begin{proof}
Define the homomorphism on idempotents by
\begin{equation*}
\rho_p(e_{i,p}) =
\begin{cases}
e_i & \text{ if } i \leq p \\
e_{i+1} & \text{ if } i >p.
\end{cases}
\end{equation*}
On "dot" generators define it by
\begin{equation*}
\rho_p(y_j e_{i,p}) =
\begin{cases}
y_j e_{i+1} & \text{ if } i > p, j < p \\
y_{j+1} e_{i+1} & \text{ if } i >p, j>p \\
y_j e_{i} & \text{ if } i \leq p, j < p+1 \\
y_{j+1} e_{i} & \text{ if } i \leq p, j > p+1
\end{cases}
\end{equation*}
\begin{equation*}
\rho_p(E_1 e_{i,p}) =
\begin{cases}
(y_p+y_{p+1})e_{i+1} & \text{ if } i > p \\
(y_p+y_{p+1})e_i & \text{ if } i \leq p \\
\end{cases}
\end{equation*}
\begin{equation*}
\rho_p(E_2 e_{i,p}) =
\begin{cases}
(y_py_{p+1})e_{i+1} & \text{ if } i > p \\
(y_py_{p+1})e_i & \text{ if } i \leq p.
\end{cases}
\end{equation*}
On "crossing" generators
\begin{equation*}
\rho_p(\psi_i e_{i,p}) =
\begin{cases}
\psi_{i+1} e_{i+1} & \text{ if } i > p \\
\psi_{i} e_{i} & \text{ if } i < p-1
\end{cases}
\end{equation*}
\begin{equation*}
\rho_p(\psi_{p-1} e_{p-1,p})= \psi_p \psi_{p-1} e_{p-1}
\end{equation*}
\begin{equation*}
\rho_p(\psi_{i-1} e_{i,p}) =
\begin{cases}
\psi_{i} e_{i+1} & \text{ if } i > p+1 \\
\psi_{i-1} e_{i} & \text{ if } i < p
\end{cases}
\end{equation*}
\begin{equation*}
\rho_p(\psi_{p+1} e_{p+1,p})= \psi_p \psi_{p+1} e_{p+2}.
\end{equation*}

The injectivity of $\rho_p$ is clear.
\end{proof}


\section{Singular Soergel bimodules}
\label{sectionsingsoerg}
\subsection{Definitions}

Let $ I \subset \{1, \ldots, n-1\}$ and $S_I $ the subgroup of $S_n$ generated by
$s_i$ for $ i \in I$.  
We abbreviate the set $\{1, \ldots, k-1, k+1, \ldots, n-1\} $ by $\hat{k}$.
Denote by $R^{I}$ the ring of invariants of $R$ under the action of $S_I$.

The category of singular Soergel bimodules ${}^I \mathcal{R}^J$ is the smallest full subcategory of $(R^I,R^J)-\gmod$ which contains all objects isomorphic to direct summands of shifts of bimodules of the form 
\begin{equation*}
R^{I_1} \otimes_{R^{J_1}} R^{I_2} \cdots \otimes_{R^{J_{n-1}}} R^{I_n}
\end{equation*}
where
\begin{equation*}
I=I_1 \subset J_1 \supset I_2 \subset  \cdots \subset J_{n-1} \supset I_n = J.
\end{equation*}

\subsection{Results}

The indecomposable objects of ${}^I \mathcal{R}^J$ were classified by Williamson ~\cite{Wi} generalizing Soergel's result for the case $I=J=\emptyset$ which we call the "regular" case.  Theorem ~\ref{indecomsoergelconstruction} is a special case of Williamson's result.

\begin{theorem}
\cite[Theorem 7.4.2]{Wi}
\label{indecomsoergelconstruction}
Let $I= \emptyset$ and $J=\{1, \ldots, k-1,k+1, \ldots, n-1 \}$.  
For each $ c \in S_n / S_J$ there is an indecomposable self-dual singular Soergel bimodule ${}^{\emptyset} B_c^J$ where the duality functor is given by
\begin{equation*}
D(M)=\Hom_{{}^{\emptyset} \mathcal{R}^{\hat{k}}}(M, R \langle -k(k-1)-(n-k)(n-k-1) \rangle).
\end{equation*}
Furthermore, each indecomposable singular Soergel bimodule is isomorphic to a ${}^{\emptyset} B_c^J$ up to isomorphism and grading shift.
\end{theorem}

An important role in the work of Soergel and Williamson is played by standard bimodules.  For each coset $c \in S_n / S_{J}$ there is the standard bimodule $R_c$.  As a left $R$-module $R_c =R$.  The right action of $R$ is twisted.  Let $c_-$ be the minimal length element in $c$.  For $m \in R_c$ and $r \in R$ define
\begin{equation*}
m . r= m(c_- .r).
\end{equation*}

For the special case $J=\hat{1}=\{2,\ldots, n-1 \}$ we will denote
$ {}^{\emptyset} B_{s_i \cdots s_1}^J $ by $ P_{i}$.  By definition $P_0 = {}^{\emptyset} B_{e}^J \cong R^J$.

We will make heavy use of special bimodules $B_i= R \otimes_{R^{\{i\}}} R \langle -1 \rangle$.

\begin{prop}
\label{standardfiltBi}
There is a short exact sequence of $(R,R)$-bimodules
\begin{equation*}
0 \rightarrow R_{}  \rightarrow B_i  \langle -1 \rangle \rightarrow R_{s_i} \rightarrow 0.
\end{equation*}
\end{prop}

\begin{proof}
Let $A$ be the $(R,R)$-subbimodule of $B_i  \langle -1 \rangle$ generated by $v=x_i \otimes 1 - 1 \otimes x_{i+1}$.
It is easy to check that $x_j.v=v.x_j$ for $j=1,\ldots,n$.
Thus $A \cong R$.

In the quotient $B_i / A$ we have 
\begin{equation*}
x_i . (1 \otimes 1) - (1 \otimes 1).x_{i+1}=x_i \otimes 1- 1 \otimes x_{i+1}=0.
\end{equation*}
\end{proof}

\begin{prop}
\label{regularindecomposable}
The regular Soergel $(R,R)$-bimodule $ B_i \otimes_R B_{i-1} \otimes_R \cdots \otimes_R B_1$ is indecomposable.
\end{prop}

\begin{proof}
From Soergel's homomorphism formula and a computation in the Hecke algebra, it follows that the endomorphism algebra of this bimodule is one-dimensional in degree zero implying that the bimodule is indecomposable.
\end{proof}

We will now focus on the case $I=\emptyset$ and $J=\{2,\ldots,n-1\}$.

\begin{lemma}
\label{restrictingstandard}
For $i \geq 2$, the $(R,R)$-bimodule $R_{s_i}$ is isomorphic to $R$ view as an $(R,R^{\hat{1}})$-bimodule.
\end{lemma}

\begin{proof}
If $f$ is a symmetric function in variables $x_2,\ldots,x_{n}$ then $s_i . f =f$.
Thus as a right $R^{\hat{1}}$-module there is no twisting and $R_{s_i} \cong R$.
\end{proof}

\begin{remark}
By Lemma ~\ref{restrictingstandard}, the $(R,R)$-bimodule $R_{s_i}$ for $i\geq 2$ restricts as an $(R,R^{\hat{1}})$-bimodule to $R$.  Thus
for $i \geq 2$ the exact sequence of $(R,R)$-bimodules
\begin{equation*}
0 \rightarrow R_{}  \rightarrow B_i \langle -1 \rangle \rightarrow R_{s_i} \rightarrow 0
\end{equation*}
gives a split exact sequence of $(R,R^{\hat{1}})$-bimodules 
since by ~\cite{Wi} there are no self-extensions of $R$.
\end{remark}





We will now explicitly construct the $(R,R^{\hat{1}})$-bimodule $P_i$.  This is possible because we are working in the very special case of $J=\hat{1}=  \{2,\ldots,n-1\}$.
In the definitions and facts about the concrete $P_i$ which follow, we will ignore that the abstract $P_i$ satisfy the properties which we now list.
Define
\begin{equation}
\label{Pidef}
P_{i} = R[y]/(y-x_1)\cdots(y-x_{i+1}) \langle -i \rangle
\end{equation}
with the left $R$-module structure given by
\begin{equation*}
x_j . y^{r}=x_j y^{r}.
\end{equation*}
The right $R^{\hat{1}}$-module structure is more complicated to describe.
Set
\begin{equation*}
y^{r} . x_1= y^{r +1} 
\end{equation*}
and for $1 \leq j \leq n-1$
\begin{equation*}
y^{r} . E_j(x_2,\ldots,x_n) =y^{r}
(  (-1)^j y^j + (-1)^{j-1} E_1(x_1,\ldots,x_n) y^{j-1} + \cdots + (-1)^0E_j(x_1, \ldots, x_n) y^0).
\end{equation*}

As a $\Bbbk$-module, $P_i$ is graded with the degree of $y^{r}$ equal to $2r - i$.

\begin{prop}
$P_i$ is an indecomposable graded $(R,R^{\hat{1}})$-bimodule.
\end{prop}

\begin{proof}
It is clear from the definition of $P_i$ that it is an $(R,R^{\hat{1}})$-bimodule.
It is indecomposable because it is cyclically generated in degree zero by a vector $1$.
\end{proof}

The next proposition is clear.

\begin{prop}
\label{sesforproj}
There is a short exact sequence of $(R,R^{\hat{1}})$-bimodules:
\begin{equation*}
0 \rightarrow R[y]/(y-x_1) \cdots (y-x_{i}) \langle 2 \rangle \rightarrow
R[y]/(y-x_1) \cdots (y-x_{i+1}) \rightarrow
R[y]/(y-x_{i+1}) \rightarrow 0
\end{equation*}
where the inclusion is given by mapping $ 1 \mapsto y-x_{i+1}$.
\end{prop}

\begin{corollary}
\label{standfiltration}
The module $P_i$ has a filtration
\begin{equation*}
M_0 \subset M_1 \subset \cdots \subset M_i=P_i
\end{equation*}
such that $M_j/M_{j-1}$ is isomorphic to the standard module $R^{\hat{1}}_{s_j \cdots s_1} \langle i-2j  \rangle$.
\end{corollary}

\begin{proof}
Using the exact sequence in Proposition ~\ref{sesforproj} we may inductively build the filtration of $P_i$.
The only non-trivial fact to check is that the quotient bimodule $R[y](y-x_{i+1})$ is isomorphic (up to a shift) to the standard bimodule $R^{\hat{1}}_{s_i \cdots s_1}$.

The two objects are clearly both isomorphic as left $R$-modules to $R$.
Now to check that the isomorphism is compatible with the right $R^{\hat{1}}$ action.

On the element $1$ in the quotient $R[y]/(y-x_{i+1})$ we have $1.x_1=x_{i+1}$.
On the element $1$ in the standard bimodule we have $1.{x_1}=1(s_{i} \cdots s_1.x_1)=x_{i+1}$.  Thus the isomorphism is compatible with the right action of $x_1$.

On the element $1$ in the quotient $R[y]/(y-x_{i+1})$ we have
\begin{equation}
\label{filtr1}
1.E_j(x_2, \ldots, x_n)=(-1)^j x_{i+1}^j +(-1)^{j-1} x_{i+1}^{j-1} E_1(x_1, \ldots, x_n)
+ \cdots + (-1)^0 x_{i+1}^0 E_j(x_1, \ldots, x_n).
\end{equation}
In the standard bimodule $E_j(x_2, \ldots, x_n)$ acts on $1$ by
\begin{equation}
\label{filtr2}
(s_i \cdots s_1).E_j(x_2, \ldots, x_n).
\end{equation}
Thus we have to check that the right hand sides of ~\eqref{filtr1} and ~\eqref{filtr2} are the same.  After applying $s_1 \cdots s_i$ to both of these quantities, we must only check in the polynomial ring $R$ that
\begin{equation}
\label{filtr3}
E_j(x_2, \ldots, x_n)=(-1)^j x_{1}^j +(-1)^{j-1} x_{1}^{j-1} E_1(x_1, \ldots, x_n)
+ \cdots + (-1)^0 x_{1}^0 E_j(x_1, \ldots, x_n).
\end{equation}
This easily follows by induction on $j$ using the fact that
\begin{equation*}
E_{j+1}(x_1, \ldots, x_n)=E_{j+1}(x_2, \ldots, x_n)+x_1 E_j(x_2, \ldots, x_n).
\end{equation*}
\end{proof}

\begin{ex}
\label{producingP_ibasecase}
We will show that as an $(R,R^{\hat{1}})$-bimodule 
\begin{equation*}
B_2 \otimes_R B_1 \cong B_2 \otimes_R P_1 \cong P_2 \oplus P_0.
\end{equation*}
There is an inclusion
\begin{equation*}
R \rightarrow B_2 \otimes_R B_1 \cong R \otimes_{R^{\{2\}}} R \otimes_{R^{\{1\}}} R  \langle -2 \rangle
\end{equation*}
given by
\begin{equation*}
1 \mapsto 1 \otimes x_1 \otimes 1 - 1 \otimes 1 \otimes x_2.
\end{equation*}
There is a splitting map 
\begin{equation*}
B_2 \otimes_R B_1 \cong R \otimes_{R^{\{2\}}} R \otimes_{R^{\{1\}}} R  \langle -2 \rangle  \rightarrow R
\end{equation*}
given by
\begin{equation*}
1 \otimes 1 \otimes 1 \mapsto 0 \hspace{.5in} 1 \otimes 1 \otimes x_2 \mapsto -1.
\end{equation*}
It is an easy exercise to check that this forces
\begin{equation*}
1 \otimes 1 \otimes x_1 \mapsto 0 \hspace{.5in} 1 \otimes 1 \otimes x_3 \mapsto 1.
\end{equation*}
\end{ex}

\begin{prop}
$P_i$ is an indecomposable Soergel $(R,R^{\hat{1}})$-bimodule.
\end{prop}

\begin{proof}
In order to prove it's a Soergel bimodule we will inductively prove
\begin{equation*}
B_{i+1} \otimes_R P_i \cong P_{i+1} \oplus P_{i-1}.
\end{equation*}
The case $i=0$ is provided by the isomorphism
\begin{equation*}
B_1 \cong R[y]/(y-x_1)(y-x_2) \langle -1 \rangle \hspace{.5in} 1 \otimes 1 \mapsto 1
\hspace{.5in} 1 \otimes x_2 \mapsto x_1+x_2-y.
\end{equation*}
It is also explained in Section ~\ref{subsectionw(2,1)}.
The case $i=1$ is given in Example ~\ref{producingP_ibasecase}.

Now let us assume that $P_i$ is a Soergel bimodule.  We will prove that $P_{i+1}$ and $P_{i-1}$  are the two indecomposable summands of 
\begin{equation*}
B_{i+1} \otimes_R P_i \cong R \otimes_{R^{i+1}} R[y]/(y-x_1)\cdots(y-x_{i+1}) \langle -i-1 \rangle.
\end{equation*}
The elements $1\otimes 1$ and $1 \otimes x_{i+1}$ generate 
$ R \otimes_{R^{i+1}} R[y]/(y-x_1)\cdots(y-x_{i+1})$ as an $(R,R^{\hat{1}})$-bimodule because $(1 \otimes 1).x_1=1 \otimes y$ and
for $j \neq i+1,i+2$, we have $x_j.(1 \otimes 1)=1 \otimes x_j$.
Finally we have 
\begin{equation*}
1 \otimes x_{i+2} =1 \otimes (x_{i+1}+x_{i+2}-x_{i+1}) 
= (x_{i+1}+x_{i+2}).(1 \otimes 1)-1 \otimes x_{i+1}.
\end{equation*}

First we will demonstrate that $P_{i-1} $ is a summand of $ B_{i+1} \otimes_R P_i$.
There is an inclusion map
\begin{equation*}
\alpha \colon P_{i-1}=R[y]/(y-x_1) \cdots (y-x_i) \langle -i+1 \rangle  \rightarrow R \otimes_{R^{i+1}} R[y]/(y-x_1)\cdots(y-x_{i+1}) \langle -i-1 \rangle
\end{equation*}
determined by 
\begin{equation*}
\alpha(1)=1 \otimes (y-x_{i+1}).
\end{equation*}

There is a map
\begin{equation*}
\beta \colon R \otimes_{R^{i+1}} R[y]/(y-x_1)\cdots(y-x_{i+1}) \langle -i-1 \rangle  \rightarrow P_{i-1}=R[y]/(y-x_1) \cdots (y-x_i) \langle -i+1 \rangle
\end{equation*}
determined by
\begin{equation*}
\beta(1 \otimes 1)=0 \hspace{.5in} \beta(1 \otimes x_{i+1})=-1.
\end{equation*}
Note that this forces for $j \neq i+1,i+2$
\begin{equation*}
\beta(1 \otimes y)=0 \hspace{.5in} \beta(1 \otimes x_{i+2})=1 \hspace{.5in} \beta(1 \otimes x_{j})=0.
\end{equation*}
It is clear that $\beta \circ \alpha = \Id$.

Next we will show that $P_{i+1} $ is a summand of $ B_{i+1} \otimes_R P_i$.
There is an inclusion map
\begin{equation*}
\gamma \colon P_{i+1}=R[y]/(y-x_1) \cdots (y-x_{i+2}) \langle -i-1 \rangle \rightarrow R \otimes_{R^{i+1}} R[y]/(y-x_1)\cdots(y-x_{i+1}) \langle -i-1 \rangle
\end{equation*}
determined by 
\begin{equation*}
\gamma(1)=1 \otimes1.
\end{equation*}

There is a map
\begin{equation*}
\delta \colon R \otimes_{R^{i+1}} R[y]/(y-x_1)\cdots(y-x_{i+1})  \langle -i-1 \rangle \rightarrow P_{i+1}=R[y]/(y-x_1) \cdots (y-x_{i+2})  \langle -i-1 \rangle
\end{equation*}
determined by
\begin{equation*}
\delta(1 \otimes 1)=1 \hspace{.5in} \delta(1 \otimes x_{i+1})=y.
\end{equation*}
Note that this forces
\begin{equation*}
\delta(1 \otimes y)=y \hspace{.5in} \delta(1 \otimes x_{i+2})=x_{i+1}+x_{i+2}-y \hspace{.5in} \delta(1 \otimes x_{j})=x_j
\end{equation*}
for $j \neq i+1,i+2$.
It is clear that $\delta \circ \gamma = \Id$.

Finally, it is straightforward to check that
\[
\xymatrix{
R \otimes_{R^{i+1}} R[y]/(y-x_1)\cdots(y-x_{i+1})  \langle -i-1 \rangle \ar[rr]^{\hspace{.9in} \begin{pmatrix}  \beta \\ \delta \end{pmatrix}}
&& P_{i-1} \oplus P_{i+1}
}
\]
and
\[
\xymatrix{
P_{i-1} \oplus P_{i+1}  \ar[rr]^{ \begin{pmatrix}  \alpha & \gamma \end{pmatrix}\hspace{1.1in}} &&
R \otimes_{R^{i+1}} R[y]/(y-x_1)\cdots(y-x_{i+1})  \langle -i-1 \rangle
}
\]
are inverse maps of each other.

\end{proof}

\section{The diagrammatic Hecke category}
\label{sectionhecke}

In this section we review the diagrammatically defined Hecke category introduced in ~\cite{EK} which is equivalent (after Karoubi closure) to the category of regular Soergel bimodules for $S_n$.  In order to distinguish the pictures in this section from other pictures arising elsewhere in this paper we color the strands for these diagrams green and the strands will carry labels from the set $\{1, 2, \ldots, n-1\}$ corresponding to simple transpositions of $S_n$.

\subsection{Admissible graphs}
An admissible graph will be a finite planar graph with boundary properly embedded in $\mathbb{R} \times [0,1]$ whose edges are colored from the set $\{1,\ldots, n-1 \}$.  Only $1$-valent, $3$-valent, $4$-valent, and $6$-valent vertices are allowed, as shown in ~\eqref{vertices}.  For the $4$-valent vertices we require $|i-j|>1$.  For the $6$-valent vertices we require that $|i-j|=1$.
The boundary points naturally inherit a coloring from the edges.  
The regions of an admissible graph are decorated by polynomials $f \in \Bbbk[x_1,\ldots,x_n]$ and if two polynomials $f$ and $g$ are in a single region then we may combine them into one polynomial $fg$ in that region.

\begin{equation}
\label{vertices}
\begin{tikzpicture}

\draw (0,0) -- (0,.5)[green][thick];
\filldraw [green] (0, .5) circle (2pt);
\draw (-.25,.25) node{$i$};

\draw (1,1) -- (1,.5)[green][thick];
\filldraw [green] (1, .5) circle (2pt);
\draw (.75,.75) node{$i$};

\draw (2,0) -- (2.5,.5)[green][thick];
\draw (3,0) -- (2.5,.5)[green][thick];
\draw (2.5,.5) -- (2.5,1)[green][thick];
\draw (2,.35) node{$i$};
\draw (3,.35) node{$i$};
\draw (2.35,.85) node{$i$};

\draw (4,1) -- (4.5,.5)[green][thick];
\draw (5,1) -- (4.5,.5)[green][thick];
\draw (4.5,.5) -- (4.5,0)[green][thick];
\draw (3.95,.75) node{$i$};
\draw (5.05,.75) node{$i$};
\draw (4.35,.25) node{$i$};

\draw (6,0) -- (7,1)[green][thick];
\draw (7,0) -- (6,1)[green][thick];
\draw (6.0,.25) node{$i$};
\draw (7.0,.25) node{$j$};
\draw (6.0,.75) node{$j$};
\draw (7.0,.75) node{$i$};

\draw (8,0) -- (9,1)[green][thick];
\draw (9,0) -- (8,1)[green][thick];
\draw (8.5,0) -- (8.5,1)[green][thick];
\draw (8.0,.25) node{$i$};
\draw (8.0,.75) node{$i$};
\draw (9.0,.25) node{$i$};
\draw (9.0,.75) node{$i$};
\draw (8.4,.15) node{$j$};
\draw (8.4,.9) node{$j$};

\end{tikzpicture}
\end{equation}

\subsection{Objects and morphisms}
The objects of the Hecke category $\mathcal{H}$ are sequences $(i_1, \ldots, i_r)$ where each entry is an element of the set $\{1,\ldots,n-1\}$.

The set of morphisms between two objects $(i_1,\ldots,i_r)$ and $(j_1,\ldots ,j_s)$ is  the free $\Bbbk[x_1,\ldots,x_n]$-module generated by admissible graphs  whose boundary points contained in $\mathbb{R} \times \{0\}$ read from left to right are colored $i_1, \ldots, i_r$ and whose boundary points contained in $\mathbb{R} \times \{1\}$ read from left to right are colored $j_1, \ldots, j_s$ subject to relations we will describe in the next subsection.
A polynomial $f$ acts on a generating admissible graph by decorating the leftmost region with $f$.

The category $\mathcal{H}$ may be equipped with a $\mathbb{Z}$-grading by declaring that the first two generators in ~\eqref{vertices} have degree $1$, the next two generators have degree $-1$ and the final two generators have degree $0$.
Additionally, polynomial generators $x_i$ decorating the regions have degree $2$.

\subsection{Relations between morphisms}

We have the following polynomial slide relations:
\begin{equation}
\label{relation3.1,3.2,3.3}

\end{equation}
where in the last relation we assume $j \neq i, i+1$.
\begin{equation}
\label{relation3.6}
%
\end{equation}

\subsection{A table of strand types}
Below is a chart which gives a dictionary between the various types of strands used in the paper when viewed in color versus black and white.  Green strands typically carry labels.

\begin{equation}
\label{conversion}
%
\end{equation}

\section{A connection between $W(n,1)\-gmod$ and $ {}^{\emptyset} \mathcal{R}^{\hat{1}}$}
\label{sectionequivalence}

\subsection{General case}
Recall the construction of the indecomposable bimodule $P_i$ in the Soergel category $ {}^{\emptyset} \mathcal{R}^{\hat{1}}$ given in ~\eqref{Pidef}.
Define maps
\begin{align*}
\phi_{i+1,i} &\colon P_i \rightarrow P_{i+1}  \hspace{.5in} 1 \mapsto x_{i+1}-y \\
\phi_{i,i+1} &\colon P_{i+1} \rightarrow P_{i} \hspace{.5in} 1 \mapsto 1 \\
\lambda_{x_j,i} &\colon P_i \rightarrow P_i \hspace{.63in} 1 \mapsto x_j.
\end{align*}

\begin{theorem}
\label{theoremendalgebra}
There is an isomorphism of algebras
\begin{equation*}
\Phi \colon \End_{{}^{\emptyset} \mathcal{R}^{\hat{1}}} (\bigoplus_{i=0}^{n-1} P_i) \cong W(n,1).
\end{equation*}
\end{theorem}

\begin{proof}
Define the map $\Phi$ as follows:
\begin{align*}
\Phi(\Id_{P_i}) &= e_{i+2}, \\
\Phi(\phi_{i+1,i}) & = \psi_{i+2} e_{i+2}, \\
\Phi(\phi_{i-1,i}) & = \psi_{i+1} e_{i+2}, \\
\Phi(\lambda_{x_j,i}) &= y_j e_{i+2} \hspace{.5in} \textrm{ if } j \leq i+1, \\
\Phi(\lambda_{x_j,i}) &= y_{j+1} e_{i+2} \hspace{.35in} \textrm{ if } j > i+1. \\
\end{align*}
It is straightforward to check that $\Phi$ is a homomorphism.
The map $\Phi$ is surjective since all of the generators of $W(n,1)$ are in the image of the homomorphism.

A calculation in the Hecke algebra along with ~\cite[Theorem 7.2.2]{Wi} give the graded dimension of 
$\Hom_{{}^{\emptyset} \mathcal{R}^{\hat{1}}}(P_i,P_j) $.
Using the basis described in Proposition ~\ref{basisprop} we may easily calculate the graded dimension of $e_i W(n,1) e_j$.
Assembling these results shows
\begin{equation*}
\gdim_{\Bbbk} \Hom_{{}^{\emptyset} \mathcal{R}^{\hat{1}}}(P_i, P_j) =\gdim_{\Bbbk} e_i W(n,1) e_j,
\end{equation*}
so $\Phi$ is injective and hence an isomorphism.
\end{proof}

\begin{corollary}
The center of the category $ {}^{\emptyset} \mathcal{R}^{\hat{1}}$ is isomorphic to
$R[z]/(z-x_1)\cdots (z-x_n)$.
\end{corollary}

\begin{proof}
The calculation of the center from Proposition ~\ref{centerofW(n,1)} along with Theorem ~\ref{theoremendalgebra} give the center of the singular Soergel category.
\end{proof}

\subsection{Special case: ${W}(2,1)$}
\label{subsectionw(2,1)}
Throughout this section green (faded) strands carry the label $1$ which we omit.
In this case the category of singular Soergel bimodules $\mathcal{R}$ is just the case of the category of  "regular" Soergel bimodules for $R=\Bbbk[x_1,x_2]$.
This in turn has a graphical description which is a special case of the calculus developed in ~\cite{EK}.  

Let $R^1=\Bbbk[x_1,x_2]^{S_2}$ and $ B = R \otimes_{R^1} R \langle -1 \rangle$.
The two indecomposable objects in the category $\mathcal{R}$ are $R$ and $B$.

There is a short exact sequence
\begin{equation*}
0 \rightarrow R \langle 1 \rangle \rightarrow B \rightarrow B / R \langle 1 \rangle \rightarrow 0, 
\end{equation*}
where the inclusion is given by $1 \mapsto x_1 \otimes 1 - 1 \otimes x_2$.
We may identify the quotient bimodule $ B / R \langle 1 \rangle$ with $R_{s_1} \langle -1 \rangle$
where $R_{s_1}$ as a left $R$-module is just $R$ but the right action is twisted as follows:
\begin{equation*}
m.r = m(s_1.r).
\end{equation*}
The isomorphism $ B / R \langle 1 \rangle \cong R_{s_1} \langle -1 \rangle$ is given by mapping 
$f \otimes 1$ to $f$.

We could find an isomorphism $P_1 \cong B$, where we recall $P_1 = R[y]/(y-x_1)(y-x_2) \langle -1 \rangle $. The isomorphism is given by mapping
\begin{align*}
1 &\mapsto 1 \otimes 1\\
y &\mapsto 1 \otimes x_1.
\end{align*}
The left $R$-module structure is given by
\begin{equation*}
x_i . y^{j}=x_i y^{j}.
\end{equation*}
The right $R$-module structure is given by
\begin{align*}
y^{j}.x_1 &=  y^{j+1} \\
y^{j}.x_2 &=  y^{j}(x_1+x_2-y).
\end{align*}

Graphically we denote the identity morphism of the Soergel bimodule $B$ by a green strand.  The identity morphism of $R$ is given by just the empty diagram.

We define a functor $ \mathcal{F} \colon W(2,1)\-gmod \rightarrow \mathcal{R}$.
On objects the functor maps the projective module $W(2,1)e(\r\r\b)$ to $B$ and the projective module $W(2,1)e(\r\b\r) $ to $R$.
Recall that $e(\b\r\r)=0$.

On the level of morphisms we have

\begin{equation}
\label{map(2)}
\begin{tikzpicture}
\draw (1,0) -- (1,2)[red][thick];
\draw (2,0) -- (2,2)[red][thick];
\draw (3,0) -- (3,2)[black][thick][dashed];
\filldraw[red](1,1) circle (2pt);
\filldraw[red](2,1) circle (2pt);
\draw (.5,1) node{$a$};
\draw (1.5,1) node{$b$};

\draw (3.5,1) node{$\mapsto$};
\draw (4.5,1) node{$x_1^a x_2^b$};
\draw (5,0) -- (5,2)[green][thick];

\end{tikzpicture}
\end{equation}

\begin{equation}
\label{map(1)}
\begin{tikzpicture}
\draw (1,0) -- (1,2)[red][thick];
\draw (2,0) -- (2,2)[black][thick][dashed];
\draw (3,0) -- (3,2)[red][thick];

\filldraw[red](1,1) circle (2pt);
\filldraw[red](3,1) circle (2pt);
\draw (.5,1) node{$a$};
\draw (2.5,1) node{$b$};

\draw (3.5,1) node{$\mapsto$};

\draw (4.5,1) node{$x_1^a x_2^b$};

\end{tikzpicture}
\end{equation}

\begin{equation}
\label{map(21)}
\begin{tikzpicture}
\draw (1,0) -- (1,2)[red][thick];
\draw (2,0) -- (3,2)[black][thick][dashed];
\draw (3,0) -- (2,2)[red][thick];

\draw (3.5,1) node{$\mapsto$};

\filldraw[green](4,1) circle (2pt);
\draw (4,1) -- (4,2)[green][thick];

\end{tikzpicture}
\end{equation}

\begin{equation}
\label{map(12)}
\begin{tikzpicture}
\draw (1,0) -- (1,2)[red][thick];
\draw (2,0) -- (3,2)[red][thick];
\draw (3,0) -- (2,2)[black][thick][dashed];

\draw (3.5,1) node{$\mapsto$};

\filldraw[green](4,1) circle (2pt);
\draw (4,1) -- (4,0)[green][thick];

\end{tikzpicture}
\end{equation}

As a consequence of the relations we get
\begin{equation}
\label{conseq1}
\begin{tikzpicture}
\draw (1,0) -- (1,2)[red][thick];
\draw (2,0) -- (2,2)[black][thick][dashed];
\draw (3,0) -- (3,2)[red][thick];

\filldraw[black](2,1) circle (2pt);

\draw (3.5,1) node{$\mapsto$};

\draw (4.5,1) node{$x_1$};

\end{tikzpicture}
\end{equation}

\begin{equation}
\label{conseq2}
\begin{tikzpicture}
\draw (1,0) -- (1,2)[red][thick];
\draw (2,0) -- (2,2)[red][thick];
\draw (3,0) -- (3,2)[black][thick][dashed];
\filldraw[black](3,1) circle (2pt);

\draw (3.5,1) node{$\mapsto$};
\draw (4.5,1) node{$x_2$};
\draw (5,0) -- (5,2)[green][thick];
\draw (5.5,1) node{$+$};
\draw (6,0) -- (6,.5)[green][thick];
\draw (6,1.5) -- (6,2)[green][thick];
\filldraw[green](6,.5) circle (2pt);
\filldraw[green](6,1.5) circle (2pt);

\end{tikzpicture}
\end{equation}

The functor induces a surjective map on the spaces of morphisms. In order to prove that $\mathcal{F}$ is an equivalence we need to compute the graded dimensions of these morphism spaces.
For $\mathcal{R}$ this is well known.

\begin{align*}
\gdim_{\Bbbk} \Hom_{\mathcal{R}}(R,R) &= \frac{1}{(1-q^2)^2} \\
\gdim_{\Bbbk} \Hom_{\mathcal{R}}(B,B) &= \frac{1+q^2}{(1-q^2)^2} \\
\gdim_{\Bbbk} \Hom_{\mathcal{R}}(R,B) &= \frac{q}{(1-q^2)^2} \\
\gdim_{\Bbbk} \Hom_{\mathcal{R}}(B,R) &= \frac{q}{(1-q^2)^2}.
\end{align*}

Now it is easy to see that a spanning set of
$ \Hom_{W(2,1)}({W}(2,1)e(\r\b\r),{W}(2,1)e(\r\b\r))$ consists of diagrams in ~\eqref{span(11)} with $a,b \in \Z_{\geq 0}$
\begin{equation}
\label{span(11)}
\begin{tikzpicture}
\draw (1,0) -- (1,2)[red][thick];
\draw (2,0) -- (2,2)[black][thick][dashed];
\draw (3,0) -- (3,2)[red][thick];

\filldraw[red](1,1) circle (2pt);
\filldraw[red](3,1) circle (2pt);
\draw (.5,1) node{$a$};
\draw (2.5,1) node{$b$};

\end{tikzpicture}
\end{equation}
since the relations imply that
\begin{equation}
\begin{tikzpicture}
\draw (1,0) -- (1,2)[red][thick];
\draw (2,0) -- (2,2)[black][thick][dashed];
\draw (3,0) -- (3,2)[red][thick];

\filldraw[red](1,1) circle (2pt);

\draw (3.5,1) node{$=$};

\draw (4,0) -- (4,2)[red][thick];
\draw (5,0) -- (5,2)[black][thick][dashed];
\draw (6,0) -- (6,2)[red][thick];
\filldraw[black](5,1) circle (2pt);

\end{tikzpicture}
\end{equation}

We then easily get that $ \Hom_{W(2,1)}({W}(2,1)e(\r\b\r),{W}(2,1)e(\r\r\b))$ is spanned by diagrams given in ~\eqref{span(21)} with $a,b \in \Z_{\geq 0}$.

\begin{equation}
\label{span(21)}
\begin{tikzpicture}
\draw (1,0) -- (1,2)[red][thick];
\draw (2,0) -- (3,2)[black][thick][dashed];
\draw (3,0) -- (2,2)[red][thick];

\filldraw[red](1,1.75) circle (2pt);
\filldraw[red](2.125,1.75) circle (2pt);
\draw (.5,1.75) node{$a$};
\draw (1.75,1.75) node{$b$};

\end{tikzpicture}
\end{equation}

Similarly $ \Hom_{W(2,1)}({W}(2,1)e(\r\r\b),{W}(2,1)e(\r\b\r))$ is spanned by diagrams given in ~\eqref{span(12)} with $a,b \in \Z_{\geq 0}$.

\begin{equation}
\label{span(12)}
\begin{tikzpicture}
\draw (1,0) -- (1,2)[red][thick];
\draw (2,0) -- (3,2)[red][thick];
\draw (3,0) -- (2,2)[black][thick][dashed];

\filldraw[red](1,1.75) circle (2pt);
\filldraw[red](2.875,1.75) circle (2pt);
\draw (.5,1.75) node{$a$};
\draw (3.125,1.75) node{$b$};

\end{tikzpicture}
\end{equation}

We claim that $ \Hom_{W(2,1)}({W}(2,1)e(\r\r\b),{W}(2,1)e(\r\r\b))$ is spanned by diagrams given in ~\eqref{span(22)} with $a,b \in \Z_{\geq 0}, c \in \{0,1\}$.

\begin{equation}
\label{span(22)}
\begin{tikzpicture}
\draw (1,0) -- (1,2)[red][thick];
\draw (2,0) -- (2,2)[red][thick];
\draw (3,0) -- (3,2)[black][thick][dashed];

\filldraw[red](1,1) circle (2pt);
\filldraw[red](2,1) circle (2pt);
\filldraw[black](3,1) circle (2pt);
\draw (.5,1) node{$a$};
\draw (1.5,1) node{$b$};
\draw (2.5,1) node{$c$};

\end{tikzpicture}
\end{equation}

That the set of morphisms in ~\eqref{span(22)} is a spanning set follows from the equation

\begin{equation}
\begin{tikzpicture}
\draw (1,0) -- (1,2)[red][thick];
\draw (2,0) -- (2,2)[red][thick];
\draw (3,0) -- (3,2)[black][thick][dashed];

\filldraw[black](3,1) circle (2pt);

\draw (2.5,1) node{$2$};

\draw (3.5,1) node{$=$};

\draw (4,0) -- (4,2)[red][thick];
\filldraw[red](4,1) circle (2pt);

\draw (5,0) -- (5,2)[red][thick];
\draw (6,0) -- (6,2)[black][thick][dashed];
\filldraw[red](6,1) circle (2pt);

\draw (6.5,1) node{$+$};

\draw (7,0) -- (7,2)[red][thick];
\draw (8,0) -- (8,2)[red][thick];
\filldraw[red](8,1) circle (2pt);

\draw (9,0) -- (9,2)[black][thick][dashed];
\filldraw[black](9,1) circle (2pt);

\draw (9.5,1) node{$-$};

\draw (10,0) -- (10,2)[red][thick];
\filldraw[red](10,1) circle (2pt);

\draw (11,0) -- (11,2)[red][thick];
\filldraw[red](11,1) circle (2pt);

\draw (12,0) -- (12,2)[black][thick][dashed];

\end{tikzpicture}
\end{equation}
which comes from pulling the black strand in the identity diagram all the way to the left (which is then zero in the cyclotomic quotient) and applying double crossing relations.

There is a faithful representation of $W(2,1)$ on $R[y_1]/(y_1-x_1) \oplus R[y_2]/((y_2-x_1)(y_2-x_2)) $ described in greater generality earlier, implying
that the spanning sets above are actually bases of homomorphism spaces.

\section{$W(n,1)$ as a deformation}
\label{sectiondeformation}

\subsection{The algebra $\overline{W}(n,1)$}

Let $ Q $ denote the quiver in ~\eqref{quiversl2} with vertices $1, \ldots, n$.
A path (of length $ r-1$) is a sequence $ p = (i_r | \cdots | i_1) $ of arrows where the starting point is $ i_1 $ and the ending point is $ i_r$.
By $ \Bbbk Q $ we denote the path algebra of $ Q $, with basis the set of all paths with product given by concatenation.
The path algebra is a graded algebra where the grading comes from the length of each path.
\begin{equation}
\label{quiversl2}
\begin{tikzpicture}
\filldraw[black](0,0) circle (1pt);
\filldraw[black](1,0) circle (1pt);
\filldraw[black](2,0) circle (1pt);
\draw (3,0) node{$\cdots$};
\filldraw[black](4,0) circle (1pt);
\filldraw[black](5,0) circle (1pt);
\draw (0,0) .. controls (.5,.5) .. (1,0)[->][thick];
\draw (0,0) .. controls (.5,-.5) .. (1,0)[<-][thick];
\draw (1,0) .. controls (1.5,.5) .. (2,0)[->][thick];
\draw (1,0) .. controls (1.5,-.5) .. (2,0)[<-][thick];
\draw (4,0) .. controls (4.5,.5) .. (5,0)[->][thick];
\draw (4,0) .. controls (4.5,-.5) .. (5,0)[<-][thick];
\draw (.5,.6) node{$$};
\draw (.5,-.6) node{$$};
\draw (0,-.3) node{$1$};
\draw (1,-.3) node{$2$};
\draw (2,-.3) node{$3$};
\draw (3.75,-.3) node{$n-1$};
\draw (5,-.3) node{$n$};
\end{tikzpicture}
\end{equation}

Set $ A_{n}^! $ to be the algebra $ \Bbbk Q $ modulo the two-sided ideal generated by $ (i|i-1|i)-(i|i+1|i)$ for $ i=1,\ldots,n-1$, where by convention $ (1|0|1)=0$.
By abuse of notation, we denote the image of an element $ p \in \Bbbk Q $ in the algebra $ A_{n}^! $ also by $ p$.  
The algebra $ A_n^! $ inherits a grading from $ \Bbbk Q$ since the relations are homogenous.  Let $ (A_n^!)^j $ denote the degree $ j $ subspace of $ A_n^! $.  The degree zero part $ (A_n^!)^0 $ is a semi-simple algebra
with basis $ \{ (1), \ldots, (n) \} $.

Define $\overline{W}(n,1)$ to be the quotient of $W(n,1)$ by the two-sided ideal generated by diagrams with dots on red strands.
It is easy to see that this quotient is a finite-dimensional algebra isomorphic to $A_n^!$.

\begin{remark} The graded category $ {}^{\mathbb{Z}} \mathcal{O}_{(1,n-1)}(\mathfrak{gl}_n) $ is equivalent to the category of finitely generated, graded, right modules over the algebra $ A_n^! $.
The projective modules $ (1)A_n^! $ and $ (n)A_n^! $ correspond to the dominant and anti-dominant projective modules respectively in category $ {}^{\mathbb{Z}}\mathcal{O}_{(1,n-1)}(\mathfrak{gl}_n)$.  For the definition of graded category $\mathcal{O}$ see for example ~\cite{FKS}.
\end{remark}

\subsection{$\HH^*(\overline{W}(n,1))$}
We will determine the Hochschild cohomology of $\overline{W}(n,1)$ by studying its isomorphic algebra $A_n^!$.
Since $A_n^!$ is graded, its Hochschild cohomology is graded as well.
We will let $\HH^{i,j}(A_n^!)$ denote the subspace of the Hochschild cohomology in cohomological degree $i$ and internal degree $j$.

The center of $A_n^!$, which is isomorphic to $\HH^0(A_n^!)$, is well known.
  
\begin{prop}
As a graded algebra $ \HH^0(A_n^!) \cong \Bbbk[c]/(c^n) $ where
$ c = (2|1|2)+\cdots+(n|n-1|n) $.
Thus $\HH^{0,2j}(A_n^!) \cong \Bbbk $ for $ 0 \leq j \leq n-1$ and zero otherwise.
\end{prop}

In order to determine the higher Hochschild cohomology we construct a resolution of $A_n^!$.

\begin{prop}
\label{Aresolved}
There is a projective resolution of $A_n^!$ as a bimodule:
\[
\xymatrix{
\bigoplus_{i=1}^{n-1} A_n^!(i) \otimes (i)A_n^! \langle 2 \rangle \ar[r]^{f_2 \hspace{.1in}} 
& {\begin{matrix}  \bigoplus_{i=2}^n A_n^!(i) \otimes (i-1)A_n^! \langle 1 \rangle \\ \oplus \\ 
\bigoplus_{i=2}^n A_n^!(i-1) \otimes (i)A_n^! \langle 1 \rangle
\end{matrix}}
\ar[r]^{\hspace{.2in} f_1}
& \bigoplus_{i=1}^n A_n^!(i) \otimes (i)A_n^!
}
\]

where
\begin{align*}
f_1 \colon 
& (i) \otimes (i-1) \mapsto (i|i-1) \otimes (i-1) - (i) \otimes (i|i-1) \\
& (i-1) \otimes (i) \mapsto (i-1|i) \otimes (i) - (i-1) \otimes (i-1|i) \\ 
f_2 \colon & (i) \otimes (i) \mapsto (i|i+1) \otimes (i) + (i) \otimes (i+1|i) - (i) \otimes (i-1|i) - (i|i-1) \otimes (i).
\end{align*}
\end{prop}

\begin{corollary}
If $ i \geq 3$ then $\HH^i(A_n^!)=0$.
\end{corollary}

\begin{proof}
This follows directly from the fact that length of the resolution in Proposition ~\ref{Aresolved} is two.
\end{proof}

\begin{prop}
The first Hochschild cohomology of $A_n^!$ is given by:
\begin{equation*}
\HH^{1,j}(A_n^!) \cong
\begin{cases}
\Bbbk & \text{ if } j=0,2, \ldots, 2(n-2) \\
0 & \text{ otherwise. }
\end{cases}
\end{equation*}
\end{prop}

\begin{proof}
This is straightforward using the resolution of $A_n^!$ in Proposition ~\ref{Aresolved}.
\end{proof}

\begin{prop}
\label{HH2prop}
The second Hochschild cohomology is given by:
\begin{equation*}
\HH^{2,j}(A_n^!) \cong
\begin{cases}
\Bbbk^{n-1} & \text{ if } j=0 \\
0 & \text{ otherwise. }
\end{cases}
\end{equation*}
\end{prop}

\begin{proof}
It is easy to see that all $A_n^!$-bimodule maps $A_n^!(i) \otimes (i) A_n^! $ are in the image of $f_2^!$ except the projection maps $(i) \otimes (i) \mapsto (i)$.
\end{proof}

\subsection{The deformation}

\begin{prop}
For $i=2,\ldots,n$
the $\Bbbk$-linear map $\mu_i \colon A_n^! \otimes A_n^! \rightarrow A_n^!$ which is defined by
\begin{equation*}
\mu_i((i|i-1)\otimes(i-1|i))=-(i)
\hspace{.5in}
\mu_i((i-1|i)\otimes(i|i-1))=-(i-1)
\end{equation*}
and takes the value zero on all other basis elements is a non-trivial $2$-cocycle.  
Furthermore, these $\mu_i$ are linearly independent and span $\HH^2(A_n^!)$.
\end{prop}

\begin{proof}
The $\mu_i$ are clearly linearly independent so it follows that they are a basis by Proposition ~\ref{HH2prop}.
\end{proof}

The element $\mu_i \in \HH^2(A_n^!)$ corresponds to deforming the relation in the algebra $\overline{W}(n,1)$ 
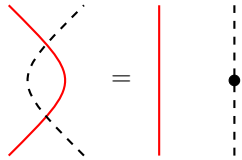
\begin{equation*}
\begin{tikzpicture}
\draw (0,0) .. controls (1,1) .. (0,2)[red][thick];
\draw (1,0) .. controls (0,1) .. (1,2)[][thick][dashed];
\draw (1.5,1) node{$=$};
\draw (2,0) -- (2,2)[red][thick];
\draw (3,0) -- (3,2)[thick][dashed];
\filldraw[black](3,1) circle (2pt);
\end{tikzpicture}
\end{equation*}
by subtracting a term with a dot on the $i$-th solid red strand:
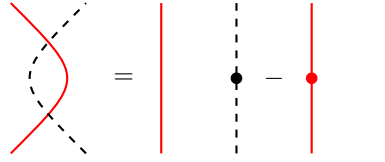
\begin{equation*}
\begin{tikzpicture}
\draw (0,0) .. controls (1,1) .. (0,2)[red][thick];
\draw (1,0) .. controls (0,1) .. (1,2)[][thick][dashed];
\draw (1.5,1) node{$=$};
\draw (2,0) -- (2,2)[red][thick];
\draw (3,0) -- (3,2)[thick][dashed];
\filldraw[black](3,1) circle (2pt);
\draw (3.5,1) node{$-$};
\draw (4,0) -- (4,2)[red][thick];
\draw (5,0) -- (5,2)[thick][dashed];
\filldraw[red](4,1) circle (2pt);
\end{tikzpicture}
\end{equation*}

\section{A braid group action on $Kom({W}(n,1)\-gmod)$}
\label{sectionbraid}

\subsection{Algebraic description}
Throughout this section let $W=W(n,1)$.
Recall the definition of $W^i:=W^i(n,1) \subset W(n,1)$ given in Section ~\ref{specialsubalgebra}.  
For $i=1,\ldots,n-1$, define the $(W,W)$-bimodule 
\begin{equation*}
W_i := W \otimes_{W^i} W \langle -1 \rangle.
\end{equation*}

\subsection{Graphical description of $W_i$}
We would like to give a graphical description of the bimodule $W_i$ for $i=1,\ldots,n-1$.
We consider collections of smooth arcs in the plane connecting $ n $ red points and $1$ black point on one horizontal line with $n$ red points and $1$ black point on another horizontal line.
The $i$th and $(i+1)$st red dots on one horizontal line must be connected to the
 $i$th and $(i+1)$st red dots on the other horizontal line by a diagram which has a thick red strand in the middle, which is given in ~\eqref{thickgenerator}.
The arcs are colored in a manner consistent with their boundary points and the black arc is dashed.
Arcs are assumed to have no critical points (in other words no cups or caps).
Arcs are allowed to intersect, but no triple intersections are allowed.
Arcs are allowed to carry dots.  
Two diagrams that are related by an isotopy that does not change the combinatorial types of the diagrams or the relative position of crossings are taken to be equal up to sign.
The elements of the vector space $ W_i $ are formal linear combinations of these diagrams modulo the local relations for $W$ along with the relations given in
~\eqref{symmetricrelations1}, ~\eqref{symmetricrelations2}, ~\eqref{blackthickred1}, and ~\eqref{blackthickred2}.


\begin{equation}
\label{thickgenerator}

\end{equation}
The vector space $W_i$ is naturally a $(W,W)$-bimodule, and
\begin{equation*}
W_i = \bigoplus_{k=2}^{n+1} \bigoplus_{l=2}^{n+1} e_k W_i e_l
\end{equation*}
as a vector space.

\subsection{Bimodule homomorphisms}
\subsubsection{The map $\epsilon_i$}

We now aim to define the \textit{lollipop} map $\epsilon_i$.
There is a bimodule map $ \epsilon_i \colon W_i \rightarrow W$ determined by

\begin{equation}
\label{bimodmap1}
%
\end{equation}

\subsubsection{The map $\iota_i$}
Now we aim to define the \textit{upside down lollipop} map.
We begin with some lemmas needed to construct this map.

\begin{lemma}
\label{bimodulelemma}
There is an equality
\begin{equation}
%
\end{equation}
By the basic graphical relations, the left hand side of ~\eqref{deg4a} is equal to
\begin{equation}
\label{deg4b}
%
\end{equation}
In a similar fashion one proves
\begin{equation}
\label{deg4d}
%
\end{equation}
Equations ~\eqref{deg4c} and ~\eqref{deg4d} imply the lemma.
\end{proof}

\begin{lemma}
As a left $R$-module, the space $e_{i+1} W_i e_{i+1} $ is spanned by the following three types of elements

\begin{equation}
\label{Bibasis}
%
\end{equation}
for $ 0 \leq a \leq i-1$.
We have omitted vertical red solid strands in all of the diagrams and the $a$ on top of a dot means there are $a$ dots.
\end{lemma}

\begin{proof}
We must check that the following element is in the left $R$-span of ~\eqref{Bibasis}

\begin{equation}
%
\end{equation}

By adding and subtracting terms we have:

\begin{equation}
%
\end{equation}

Lemma ~\ref{bimodulelemma} implies the first two terms on the right hand side of the equation above are in the $R$-span.  
Relation \eqref{symmetricrelations1} implies the third term is in the $R$-span.

\end{proof}

There is a map (which we will next show is well-defined) $\iota_i \colon W \rightarrow W_i $ determined by

\begin{equation}
\label{bimodmap2}
%
\end{equation}

\begin{prop}
Equations ~\eqref{bimodmap2} and ~\eqref{bimodmap2a} define a bimodule map $\iota_i \colon W \rightarrow W_i $.
\end{prop}

\begin{proof}
It needs to be shown that for any $r \in W$ 
\begin{equation}
\label{equivariantcheck1}
r \iota_i(1) = \iota_i(1) r,
\end{equation}
since ~\eqref{bimodmap2} and ~\eqref{bimodmap2a} describe $\iota_i(1) \in W_i $.
We check that ~\eqref{equivariantcheck1} is satisfied for all generators $r$ of $W$.  

If $r=e_j$ is an idempotent then ~\eqref{equivariantcheck1} obviously holds.

If $r$ is a diagram without any crossings and only a single dot on some strand then it suffices to check ~\eqref{equivariantcheck2} instead of ~\eqref{equivariantcheck1}.
\begin{equation}
\label{equivariantcheck2}
r \iota_i(e_j) = \iota_i(e_j) r.
\end{equation}
If $j \neq i+1$ then this is a classical fact.  See for example ~\cite[Section 2.3]{EK}.
If $j=i+1$, then it is easy to see that ~\eqref{equivariantcheck2} holds if the dot is on any strand other than the $i$-th (red) solid strand or $(i+1)$-st (red) solid strand.
We will now check the case that the dot occurs on the $i$-th solid strand.  
This means that we must verify ~\eqref{equivariantcheck1} for $r$ given below.
\begin{equation}

\end{equation}
This means that we need to check
\begin{equation}
\label{needtocheck}
%
\end{equation}

By adding and subtracting the same term we have

\begin{equation}
\label{addsubstract}
%
\end{equation}
By pairing the first and third terms in ~\eqref{addsubstract} and the second and fourth terms using relations for the bimodule $W_i$ and Lemma ~\ref{bimodulelemma} we get that ~\eqref{addsubstract} is equal to

\begin{equation}
%
\end{equation}

One proves in a similar way that ~\eqref{equivariantcheck1} holds for the element $r$ given below.
\begin{equation*}
%
\end{equation*}

Finally we have to check ~\eqref{equivariantcheck1} when the diagram representing $r$ has a single crossing with no dots.  There are four non-trivial cases.

Case 1:
\begin{equation}
\label{iotacrossing1a}
%
\end{equation}
The right hand sides of ~\eqref{iotacrossing1a} and ~\eqref{iotacrossing1b} are the same so the first case is verified.

Case 2:

\begin{equation}
\label{iotacrossing2a}
%
\end{equation}
The right hand sides of ~\eqref{iotacrossing2a} and ~\eqref{iotacrossing2b} are identical so the second case is verified.

Case 3:
\begin{equation}
\label{iotacrossing3a}
%
\end{equation}
By relation ~\eqref{symmetricrelations1} of the bimodule, the right hand sides of ~\eqref{iotacrossing3a} and ~\eqref{iotacrossing3b} are equal.

Case 4:

\begin{equation}
\label{iotacrossing4a}
%
\end{equation}
By relation ~\eqref{symmetricrelations1} of the bimodule, the right hand sides of ~\eqref{iotacrossing4a} and ~\eqref{iotacrossing4b} are equal.
\end{proof}

\subsubsection{The map \ding{116}$_i$    }

Define the \textit{trivalent map} of $W$-bimodules 
\ding{116}$_i \colon W_i \rightarrow W_i \otimes_{W} W_i $ by:
\begin{equation}
\label{Ymap}
%
\end{equation}
Relations in $W_i$ (other than those from $W$) are of the form 
$a \otimes 1 - 1\otimes a$ where $a\in W^i$.
This gets mapped to $a \otimes 1 \otimes 1 - 1 \otimes 1 \otimes a$ which is zero in 
$W_i \otimes_{W} W_i$.  Thus the map is well-defined.

\subsubsection{The map \ding{115}$_i$  }
Define the other \textit{trivalent map} \ding{115}$_i \colon W_i \otimes_{W} W_i \rightarrow W_i$ by:
\begin{equation}
\label{upsidedownYmap}
%
\end{equation}

\begin{prop}
The assignment \ding{115}$_i$ in ~\eqref{upsidedownYmap} is a well defined $W$-bimodule map.
\end{prop}

\begin{proof}
It is clear that under \ding{115}$_i$, the images of both elements in ~\eqref{delta-1} in $W_i$ are the same.
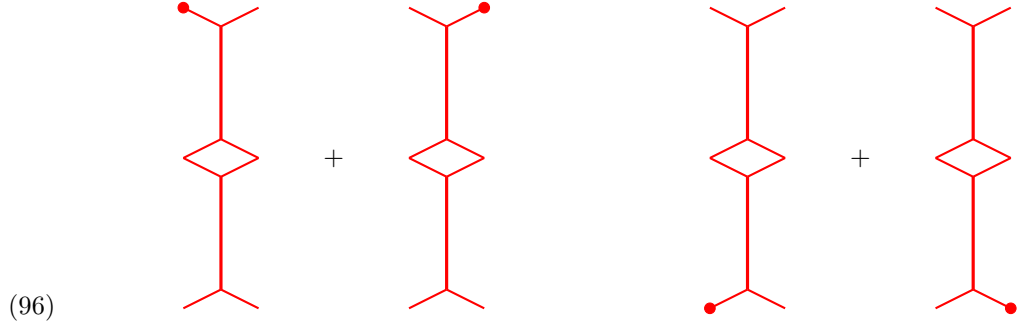
\begin{equation}
\label{delta-1}
%
\end{equation}

Similarly, the images of both elements in ~\eqref{delta0} are the same under \ding{115}$_i$.

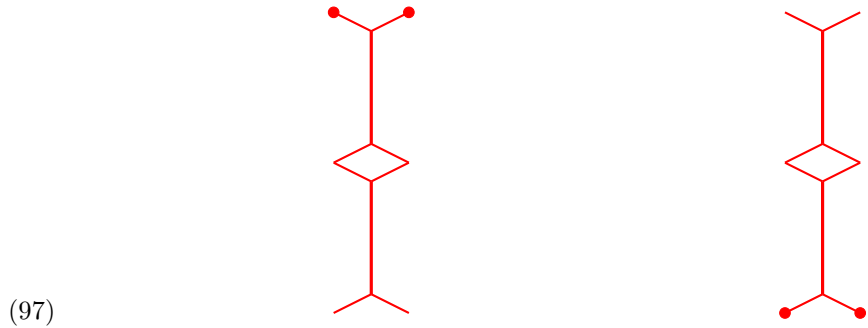
\begin{equation}
\label{delta0}
%
\end{equation}

It is clear that all diagrams in ~\eqref{delta1} (for $a=0,1$) get mapped to the same element under \ding{115}$_i$.
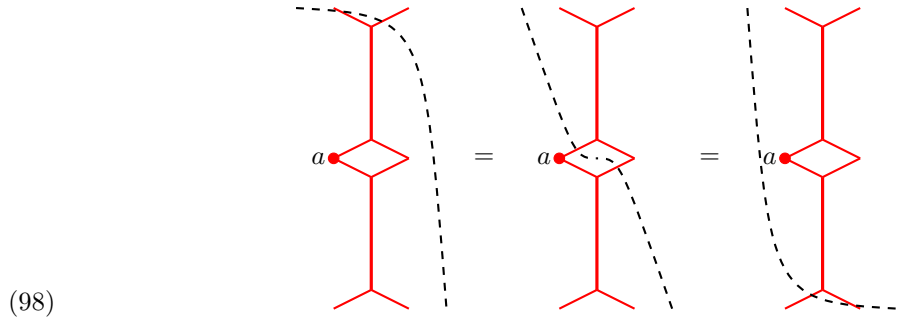
\begin{equation}
\label{delta1}
%
\end{equation}
Similarly all of the diagrams in ~\eqref{delta2} (for $a=0,1$), get mapped to the same element.
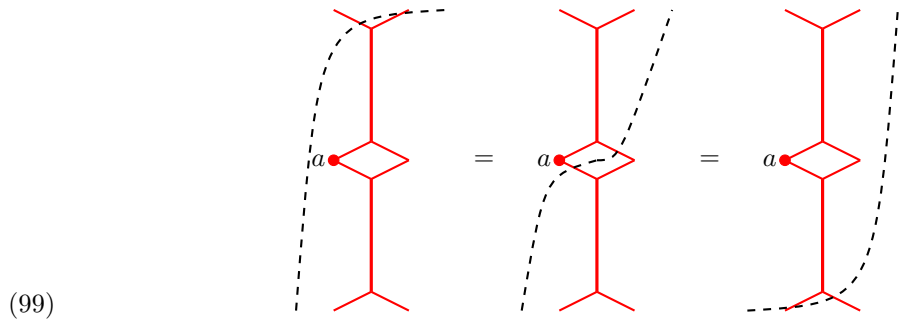
\begin{equation}
\label{delta2}
%
\end{equation}
Thus \ding{115}$_i$ is a well-defined bimodule homomorphism.
\end{proof}

\begin{lemma}
There is an isomorphism of $W$-bimodules
\begin{equation*}
W_i \otimes_{W} W_i \cong W_i \langle -1 \rangle  \oplus W_i \langle 1 \rangle.
\end{equation*}
\end{lemma}

\begin{proof}
This is proved just as in the classical base for Soergel bimodules.
One maps the element
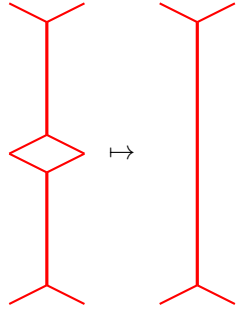
\begin{equation*}
%
\end{equation*}
in the first component and the element

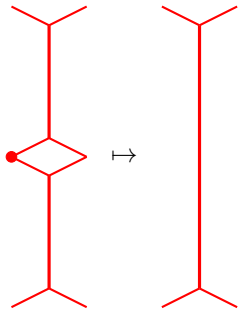
\begin{equation*}
%
\end{equation*}
in the second coordinate.

These define bimodule maps because using relations in $W_i$, one could move black strands away from the local picture.
\end{proof}


\subsubsection{The maps $X_{ij}$}

For $|i-j|>1$ we define a map 
$X_{ij} \colon W_i \otimes_{W} W_j \rightarrow W_j \otimes_{W} W_i$
determined by
\begin{equation}
\label{mapXij}

\end{equation}

\subsubsection{The maps \ding{91}$_{i,j,i}^{j,i,j}$}

For $|i-j|=1$ there will be maps  \ding{91}$_{i,j,i}^{j,i,j}$.

As a $W$-bimodule there are four generators of 
$ W_i \otimes_{W}  W_{i+1} \otimes_{W}  W_i$.  For $a=0,1$, there are two diagrams with no dashed lines interfering with the picture given in ~\eqref{noblackinterference2}.

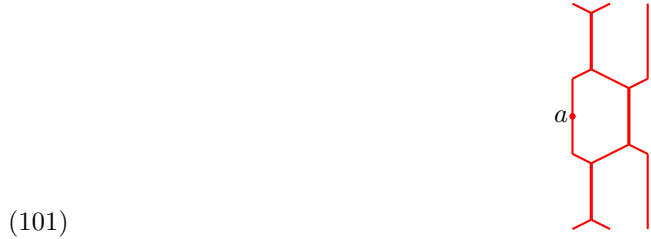
\begin{equation}
\label{noblackinterference2}
%
\end{equation}

There are two generators given by diagrams with black strand that cannot be removed locally from
$ W_i \otimes_{W}  W_{i+1} \otimes_{W}  W_i$ where again $a=0,1$.
They are given in ~\eqref{blackinterference2}.

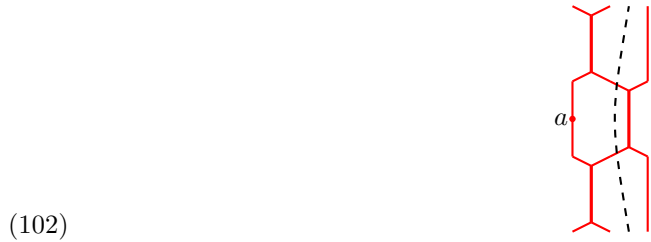
\begin{equation}
\label{blackinterference2}
%
\end{equation}

As a $W$-bimodule there are four generators of 
$ W_{i+1} \otimes_{W}  W_{i} \otimes_{W}  W_{i+1}$.  For $a=0,1$, there are two diagrams with no dashed lines interfering with the picture given in ~\eqref{noblackinterference4}.
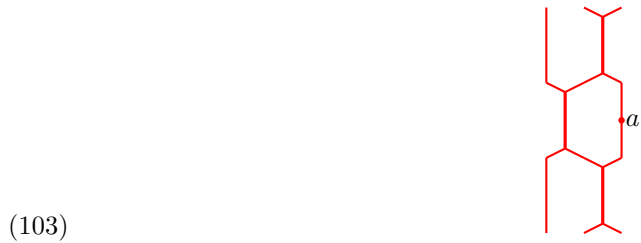
\begin{equation}
\label{noblackinterference4}
%
\end{equation}
There are diagrams with black strand that cannot be removed locally from
$ W_{i+1} \otimes_{W}  W_{i} \otimes_{W}  W_{i+1}$. They are generated (for $a=0,1$) by the diagrams given in ~\eqref{blackinterference4}.

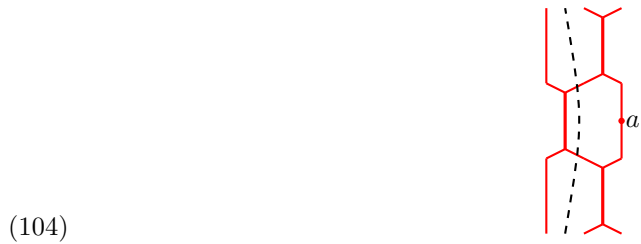
\begin{equation}
\label{blackinterference4}
%
} &&

Now we will finally define \ding{91}$_{i,i+1,i}^{i+1,i,i+1}$.
If there are no black strands in the local pictures these maps are classical and given in ~\eqref{f11a} and ~\eqref{f11b}.  The map for the diagram where there is a black dashed strand interfering with the picture is given in ~\eqref{f11c}.

\begin{equation}
\label{f11a}
%
\end{equation}

Now we will define \ding{91}$_{i+1,i,i+1}^{i,i+1,i}$.
If there are no black strands in the local pictures these maps are classical given in ~\eqref{g11a} and ~\eqref{g11b}.  The map for the diagram where there is a black dashed strand interfering with the picture is given in ~\eqref{g11c}.

\begin{equation}
\label{g11a}
%
\end{equation}

\begin{prop}
The following are bimodule homomorphisms:
\begin{enumerate}
\item  \ding{91}$_{i,i+1,i}^{i+1,i,i+1} \colon  W_i \otimes_{W}  W_{i+1} \otimes_{W}  W_i \rightarrow
W_{i+1} \otimes_{W}  W_{i} \otimes_{W}  W_{i+1}  $
\item  \ding{91}$_{i+1,i,i+1}^{i,i+1,i} \colon  W_{i+1} \otimes_{W}  W_{i} \otimes_{W}  W_{i+1} \rightarrow
W_{i} \otimes_{W}  W_{i+1} \otimes_{W}  W_{i}  $.
\end{enumerate}
\end{prop}

\begin{proof}
We will only sketch a proof of the first part of the proposition.  The second part is similar.
It must be checked that that if two elements are equal in $W_i \otimes_{W}  W_{i+1} \otimes_{W}  W_i$ via bimodule relations, then their images under 
\ding{91}$_{i,i+1,i}^{i+1,i,i+1}$ are equal.

The map \ding{91}$_{i,i+1,i}^{i+1,i,i+1}$ acts on dot generators as follows
\begin{align*}
1 \otimes x_i \otimes 1 \otimes 1 &\mapsto x_i \otimes 1 \otimes 1 \otimes 1 
+ x_{i+1} \otimes 1 \otimes 1 \otimes 1 - 1 \otimes 1 \otimes 1 \otimes x_{i+2} \\
1 \otimes x_{i+1} \otimes 1 \otimes 1 &\mapsto 1 \otimes 1 \otimes 1 \otimes x_{i+2} \\
1 \otimes x_{i+2} \otimes 1 \otimes 1 &\mapsto x_{i+2} \otimes 1 \otimes 1 \otimes 1 \\
1 \otimes 1 \otimes x_i \otimes 1 &\mapsto x_i \otimes 1 \otimes 1 \otimes 1 
+ x_{i+1} \otimes 1 \otimes 1 \otimes 1 - 1 \otimes 1 \otimes 1 \otimes x_{i+2} \\
1 \otimes 1 \otimes x_{i+1} \otimes 1 &\mapsto x_{i+2} \otimes 1 \otimes 1 \otimes 1 \\
1 \otimes 1 \otimes x_{i+2} \otimes 1 &\mapsto 1 \otimes 1 \otimes 1 \otimes x_{i+2}.
\end{align*}
Using these formulas it is easy to see that a symmetric polynomial in $x_i$ and $x_{i+1}$ could be moved between the first and second tensor factors as well as between the third and fourth tensor factors.
Similarly,  a symmetric polynomial in $x_{i+1}$ and $x_{i+2}$ could be moved between the second and third tensor factors.

Now we analyze relations involving dashed (black) strands.

It is easy to see that the elements in ~\eqref{welldefined1} both get mapped to ~\eqref{welldefined2}.

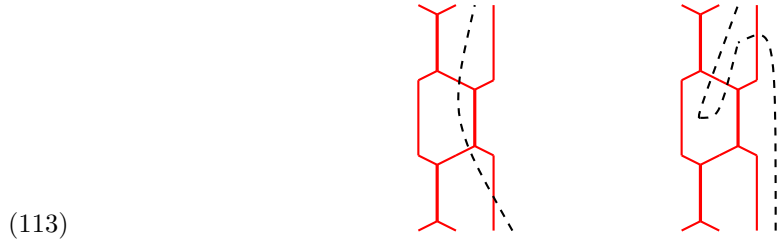
\begin{equation}
\label{welldefined1}
\begin{tikzpicture}
[scale=0.50]

\draw (0,0) -- (.5,.25)[red][thick];
\draw (1,0) -- (.5,.25)[red][thick];
\draw (0,2) -- (.5,1.75)[red][thick];
\draw (1,2) -- (.5,1.75)[red][thick];
\draw (2,0) -- (2,2)[red][thick];
\draw (.5,.25) -- (.5,1.75)[red][very thick];

\draw (1,2) -- (1.5,2.25)[red][thick];
\draw (2,2) -- (1.5,2.25)[red][thick];
\draw (1,4) -- (1.5,3.75)[red][thick];
\draw (2,4) -- (1.5,3.75)[red][thick];
\draw (0,2) -- (0,4)[red][thick];
\draw (1.5,2.25) -- (1.5,3.75)[red][very thick];

\draw (0,4) -- (.5,4.25)[red][thick];
\draw (1,4) -- (.5,4.25)[red][thick];
\draw (0,6) -- (.5,5.75)[red][thick];
\draw (1,6) -- (.5,5.75)[red][thick];
\draw (2,4) -- (2,6)[red][thick];
\draw (.5,4.25) -- (.5,5.75)[red][very thick];
\draw (2.5,0) .. controls (.8,3) .. (1.5,6)[black][thick][dashed];


\draw [shift={(7,0)}](0,0) -- (.5,.25)[red][thick];
\draw [shift={(7,0)}](1,0) -- (.5,.25)[red][thick];
\draw [shift={(7,0)}](0,2) -- (.5,1.75)[red][thick];
\draw [shift={(7,0)}](1,2) -- (.5,1.75)[red][thick];
\draw [shift={(7,0)}](2,0) -- (2,2)[red][thick];
\draw [shift={(7,0)}](.5,.25) -- (.5,1.75)[red][very thick];

\draw [shift={(7,0)}](1,2) -- (1.5,2.25)[red][thick];
\draw [shift={(7,0)}](2,2) -- (1.5,2.25)[red][thick];
\draw [shift={(7,0)}](1,4) -- (1.5,3.75)[red][thick];
\draw [shift={(7,0)}](2,4) -- (1.5,3.75)[red][thick];
\draw [shift={(7,0)}](0,2) -- (0,4)[red][thick];
\draw [shift={(7,0)}](1.5,2.25) -- (1.5,3.75)[red][very thick];

\draw [shift={(7,0)}](0,4) -- (.5,4.25)[red][thick];
\draw [shift={(7,0)}](1,4) -- (.5,4.25)[red][thick];
\draw [shift={(7,0)}](0,6) -- (.5,5.75)[red][thick];
\draw [shift={(7,0)}](1,6) -- (.5,5.75)[red][thick];
\draw [shift={(7,0)}](2,4) -- (2,6)[red][thick];
\draw [shift={(7,0)}](.5,4.25) -- (.5,5.75)[red][very thick];
\draw [shift={(7,0)}](.5,3) .. controls (.4,3) .. (1.5,6)[black][thick][dashed];
\draw [shift={(7,0)}](.5,3) .. controls (1,3) .. (1.5,5)[black][thick][dashed];
\draw [shift={(7,0)}](2.5,0) .. controls (2.5,5.5) .. (1.5,5)[black][thick][dashed];

\end{tikzpicture}
\end{equation}

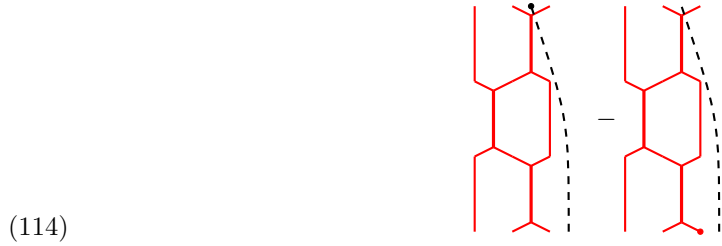
\begin{equation}
\label{welldefined2}
\begin{tikzpicture}
[scale=0.50]

\draw (0,0) -- (.5,.25)[red][thick];
\draw (1,0) -- (.5,.25)[red][thick];
\draw (0,2) -- (.5,1.75)[red][thick];
\draw (1,2) -- (.5,1.75)[red][thick];
\draw (-1,0) -- (-1,2)[red][thick];
\draw (.5,.25) -- (.5,1.75)[red][very thick];

\draw (0,2) -- (-.5,2.25)[red][thick];
\draw (-1,2) -- (-.5,2.25)[red][thick];
\draw (0,4) -- (-.5,3.75)[red][thick];
\draw (-1,4) -- (-.5,3.75)[red][thick];
\draw (1,2) -- (1,4)[red][thick];
\draw (-.5,2.25) -- (-.5,3.75)[red][very thick];

\draw (0,4) -- (.5,4.25)[red][thick];
\draw (1,4) -- (.5,4.25)[red][thick];
\draw (0,6) -- (.5,5.75)[red][thick];
\draw (1,6) -- (.5,5.75)[red][thick];
\draw (-1,4) -- (-1,6)[red][thick];
\draw (.5,4.25) -- (.5,5.75)[red][very thick];
\draw (1.5,0) .. controls (1.5,3) .. (.5,6)[black][thick][dashed];

\filldraw [black] (.5, 6) circle (2pt);

\draw (2.5,3) node{$-$};


\draw [shift={(4,0)}](0,0) -- (.5,.25)[red][thick];
\draw [shift={(4,0)}](1,0) -- (.5,.25)[red][thick];
\draw [shift={(4,0)}](0,2) -- (.5,1.75)[red][thick];
\draw [shift={(4,0)}](1,2) -- (.5,1.75)[red][thick];
\draw [shift={(4,0)}](-1,0) -- (-1,2)[red][thick];
\draw [shift={(4,0)}](.5,.25) -- (.5,1.75)[red][very thick];

\draw [shift={(4,0)}](0,2) -- (-.5,2.25)[red][thick];
\draw [shift={(4,0)}](-1,2) -- (-.5,2.25)[red][thick];
\draw [shift={(4,0)}](0,4) -- (-.5,3.75)[red][thick];
\draw [shift={(4,0)}](-1,4) -- (-.5,3.75)[red][thick];
\draw [shift={(4,0)}](1,2) -- (1,4)[red][thick];
\draw [shift={(4,0)}](-.5,2.25) -- (-.5,3.75)[red][very thick];

\draw [shift={(4,0)}](0,4) -- (.5,4.25)[red][thick];
\draw [shift={(4,0)}](1,4) -- (.5,4.25)[red][thick];
\draw [shift={(4,0)}](0,6) -- (.5,5.75)[red][thick];
\draw [shift={(4,0)}](1,6) -- (.5,5.75)[red][thick];
\draw [shift={(4,0)}](-1,4) -- (-1,6)[red][thick];
\draw[shift={(4,0)}] (.5,4.25) -- (.5,5.75)[red][very thick];
\draw [shift={(4,0)}](1.5,0) .. controls (1.5,3) .. (.5,6)[black][thick][dashed];

\filldraw [red] [shift={(4,0)}](1, 0) circle (2pt);

\end{tikzpicture}
\end{equation}

It is a routine but slightly lengthier calculation to verify that both elements in ~\eqref{welldefined3} get mapped to the same element.

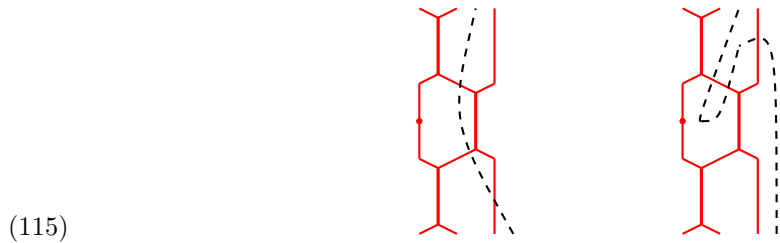
\begin{equation}
\label{welldefined3}
\begin{tikzpicture}
[scale=0.50]

\draw (0,0) -- (.5,.25)[red][thick];
\draw (1,0) -- (.5,.25)[red][thick];
\draw (0,2) -- (.5,1.75)[red][thick];
\draw (1,2) -- (.5,1.75)[red][thick];
\draw (2,0) -- (2,2)[red][thick];
\draw (.5,.25) -- (.5,1.75)[red][very thick];

\draw (1,2) -- (1.5,2.25)[red][thick];
\draw (2,2) -- (1.5,2.25)[red][thick];
\draw (1,4) -- (1.5,3.75)[red][thick];
\draw (2,4) -- (1.5,3.75)[red][thick];
\draw (0,2) -- (0,4)[red][thick];
\draw (1.5,2.25) -- (1.5,3.75)[red][very thick];

\draw (0,4) -- (.5,4.25)[red][thick];
\draw (1,4) -- (.5,4.25)[red][thick];
\draw (0,6) -- (.5,5.75)[red][thick];
\draw (1,6) -- (.5,5.75)[red][thick];
\draw (2,4) -- (2,6)[red][thick];
\draw (.5,4.25) -- (.5,5.75)[red][very thick];
\draw (2.5,0) .. controls (.8,3) .. (1.5,6)[black][thick][dashed];
\filldraw [red] (0, 3) circle (2pt);


\draw [shift={(7,0)}](0,0) -- (.5,.25)[red][thick];
\draw [shift={(7,0)}](1,0) -- (.5,.25)[red][thick];
\draw [shift={(7,0)}](0,2) -- (.5,1.75)[red][thick];
\draw [shift={(7,0)}](1,2) -- (.5,1.75)[red][thick];
\draw [shift={(7,0)}](2,0) -- (2,2)[red][thick];
\draw [shift={(7,0)}](.5,.25) -- (.5,1.75)[red][very thick];

\draw [shift={(7,0)}](1,2) -- (1.5,2.25)[red][thick];
\draw [shift={(7,0)}](2,2) -- (1.5,2.25)[red][thick];
\draw [shift={(7,0)}](1,4) -- (1.5,3.75)[red][thick];
\draw [shift={(7,0)}](2,4) -- (1.5,3.75)[red][thick];
\draw [shift={(7,0)}](0,2) -- (0,4)[red][thick];
\draw [shift={(7,0)}](1.5,2.25) -- (1.5,3.75)[red][very thick];

\draw [shift={(7,0)}](0,4) -- (.5,4.25)[red][thick];
\draw [shift={(7,0)}](1,4) -- (.5,4.25)[red][thick];
\draw [shift={(7,0)}](0,6) -- (.5,5.75)[red][thick];
\draw [shift={(7,0)}](1,6) -- (.5,5.75)[red][thick];
\draw [shift={(7,0)}](2,4) -- (2,6)[red][thick];
\draw [shift={(7,0)}](.5,4.25) -- (.5,5.75)[red][very thick];
\draw [shift={(7,0)}](.5,3) .. controls (.4,3) .. (1.5,6)[black][thick][dashed];
\draw [shift={(7,0)}](.5,3) .. controls (1,3) .. (1.5,5)[black][thick][dashed];
\draw [shift={(7,0)}](2.5,0) .. controls (2.5,5.5) .. (1.5,5)[black][thick][dashed];
\filldraw [red] [shift={(7,0)}](0, 3) circle (2pt);

\end{tikzpicture}
\end{equation}

Similarly one checks that for $a=0,1$ the elements in ~\eqref{welldefined4} get mapped to the same element with the case $a=0$ being slightly easier than $a=1$.

\begin{equation}
\label{welldefined4}
\begin{tikzpicture}
[scale=0.50]

\draw (0,0) -- (.5,.25)[red][thick];
\draw (1,0) -- (.5,.25)[red][thick];
\draw (0,2) -- (.5,1.75)[red][thick];
\draw (1,2) -- (.5,1.75)[red][thick];
\draw (2,0) -- (2,2)[red][thick];
\draw (.5,.25) -- (.5,1.75)[red][very thick];

\draw (1,2) -- (1.5,2.25)[red][thick];
\draw (2,2) -- (1.5,2.25)[red][thick];
\draw (1,4) -- (1.5,3.75)[red][thick];
\draw (2,4) -- (1.5,3.75)[red][thick];
\draw (0,2) -- (0,4)[red][thick];
\draw (1.5,2.25) -- (1.5,3.75)[red][very thick];

\draw (0,4) -- (.5,4.25)[red][thick];
\draw (1,4) -- (.5,4.25)[red][thick];
\draw (0,6) -- (.5,5.75)[red][thick];
\draw (1,6) -- (.5,5.75)[red][thick];
\draw (2,4) -- (2,6)[red][thick];
\draw (.5,4.25) -- (.5,5.75)[red][very thick];
\draw (2.5,6) .. controls (.8,3) .. (1.5,0)[black][thick][dashed];
\filldraw [red] (0, 3) circle (2pt);
\draw (-.5,3) node{$a$};


\draw [shift={(7,0)}](0,0) -- (.5,.25)[red][thick];
\draw [shift={(7,0)}](1,0) -- (.5,.25)[red][thick];
\draw [shift={(7,0)}](0,2) -- (.5,1.75)[red][thick];
\draw [shift={(7,0)}](1,2) -- (.5,1.75)[red][thick];
\draw [shift={(7,0)}](2,0) -- (2,2)[red][thick];
\draw [shift={(7,0)}](.5,.25) -- (.5,1.75)[red][very thick];

\draw [shift={(7,0)}](1,2) -- (1.5,2.25)[red][thick];
\draw [shift={(7,0)}](2,2) -- (1.5,2.25)[red][thick];
\draw [shift={(7,0)}](1,4) -- (1.5,3.75)[red][thick];
\draw [shift={(7,0)}](2,4) -- (1.5,3.75)[red][thick];
\draw [shift={(7,0)}](0,2) -- (0,4)[red][thick];
\draw [shift={(7,0)}](1.5,2.25) -- (1.5,3.75)[red][very thick];

\draw [shift={(7,0)}](0,4) -- (.5,4.25)[red][thick];
\draw [shift={(7,0)}](1,4) -- (.5,4.25)[red][thick];
\draw [shift={(7,0)}](0,6) -- (.5,5.75)[red][thick];
\draw [shift={(7,0)}](1,6) -- (.5,5.75)[red][thick];
\draw [shift={(7,0)}](2,4) -- (2,6)[red][thick];
\draw [shift={(7,0)}](.5,4.25) -- (.5,5.75)[red][very thick];
\draw [shift={(7,0)}](.5,3) .. controls (.4,3) .. (1.5,0)[black][thick][dashed];
\draw [shift={(7,0)}](.5,3) .. controls (1,3) .. (1.5,1)[black][thick][dashed];
\draw [shift={(7,0)}](2.5,6) .. controls (2.5,1) .. (1.5,1)[black][thick][dashed];
\filldraw [red] [shift={(7,0)}](0, 3) circle (2pt);
\draw [shift={(7,0)}](-.5,3) node{$a$};

\end{tikzpicture}
\end{equation}

For $a=0,1$ the images of the elements in ~\eqref{welldefined5} under \ding{91}$_{i,i+1,i}^{i+1,i,i+1}$ are the same.

\begin{equation}
\label{welldefined5}
\begin{tikzpicture}
[scale=0.50]

\draw (0,0) -- (.5,.25)[red][thick];
\draw (1,0) -- (.5,.25)[red][thick];
\draw (0,2) -- (.5,1.75)[red][thick];
\draw (1,2) -- (.5,1.75)[red][thick];
\draw (2,0) -- (2,2)[red][thick];
\draw (.5,.25) -- (.5,1.75)[red][very thick];

\draw (1,2) -- (1.5,2.25)[red][thick];
\draw (2,2) -- (1.5,2.25)[red][thick];
\draw (1,4) -- (1.5,3.75)[red][thick];
\draw (2,4) -- (1.5,3.75)[red][thick];
\draw (0,2) -- (0,4)[red][thick];
\draw (1.5,2.25) -- (1.5,3.75)[red][very thick];

\draw (0,4) -- (.5,4.25)[red][thick];
\draw (1,4) -- (.5,4.25)[red][thick];
\draw (0,6) -- (.5,5.75)[red][thick];
\draw (1,6) -- (.5,5.75)[red][thick];
\draw (2,4) -- (2,6)[red][thick];
\draw (.5,4.25) -- (.5,5.75)[red][very thick];
\draw (-.5,6) .. controls (1.3,5.9) .. (1.4,0)[black][thick][dashed];
\filldraw [red] (0, 3) circle (2pt);
\draw (-.5,3) node{$a$};


\draw [shift={(7,0)}](0,0) -- (.5,.25)[red][thick];
\draw [shift={(7,0)}](1,0) -- (.5,.25)[red][thick];
\draw [shift={(7,0)}](0,2) -- (.5,1.75)[red][thick];
\draw [shift={(7,0)}](1,2) -- (.5,1.75)[red][thick];
\draw [shift={(7,0)}](2,0) -- (2,2)[red][thick];
\draw [shift={(7,0)}](.5,.25) -- (.5,1.75)[red][very thick];

\draw [shift={(7,0)}](1,2) -- (1.5,2.25)[red][thick];
\draw [shift={(7,0)}](2,2) -- (1.5,2.25)[red][thick];
\draw [shift={(7,0)}](1,4) -- (1.5,3.75)[red][thick];
\draw [shift={(7,0)}](2,4) -- (1.5,3.75)[red][thick];
\draw [shift={(7,0)}](0,2) -- (0,4)[red][thick];
\draw [shift={(7,0)}](1.5,2.25) -- (1.5,3.75)[red][very thick];

\draw [shift={(7,0)}](0,4) -- (.5,4.25)[red][thick];
\draw [shift={(7,0)}](1,4) -- (.5,4.25)[red][thick];
\draw [shift={(7,0)}](0,6) -- (.5,5.75)[red][thick];
\draw [shift={(7,0)}](1,6) -- (.5,5.75)[red][thick];
\draw [shift={(7,0)}](2,4) -- (2,6)[red][thick];
\draw [shift={(7,0)}](.5,4.25) -- (.5,5.75)[red][very thick];
\draw [shift={(7,0)}](-.5,6) .. controls (-.4,4.5) .. (.5,3)[black][thick][dashed];
\draw [shift={(7,0)}](1.5,0) .. controls (1,5.5) .. (.5,3)[black][thick][dashed];

\filldraw [red] [shift={(7,0)}](0, 3) circle (2pt);
\draw [shift={(7,0)}](-.5,3) node{$a$};

\end{tikzpicture}
\end{equation}

For $a=0,1$ the images of the elements in ~\eqref{welldefined6} under \ding{91}$_{i,i+1,i}^{i+1,i,i+1}$ are the same.

\begin{equation}
\label{welldefined6}
\begin{tikzpicture}
[scale=0.50]

\draw (0,0) -- (.5,.25)[red][thick];
\draw (1,0) -- (.5,.25)[red][thick];
\draw (0,2) -- (.5,1.75)[red][thick];
\draw (1,2) -- (.5,1.75)[red][thick];
\draw (2,0) -- (2,2)[red][thick];
\draw (.5,.25) -- (.5,1.75)[red][very thick];

\draw (1,2) -- (1.5,2.25)[red][thick];
\draw (2,2) -- (1.5,2.25)[red][thick];
\draw (1,4) -- (1.5,3.75)[red][thick];
\draw (2,4) -- (1.5,3.75)[red][thick];
\draw (0,2) -- (0,4)[red][thick];
\draw (1.5,2.25) -- (1.5,3.75)[red][very thick];

\draw (0,4) -- (.5,4.25)[red][thick];
\draw (1,4) -- (.5,4.25)[red][thick];
\draw (0,6) -- (.5,5.75)[red][thick];
\draw (1,6) -- (.5,5.75)[red][thick];
\draw (2,4) -- (2,6)[red][thick];
\draw (.5,4.25) -- (.5,5.75)[red][very thick];
\draw (-.5,0) .. controls (1.3,.1) .. (1.4,6)[black][thick][dashed];
\filldraw [red] (0, 3) circle (2pt);
\draw (-.5,3) node{$a$};


\draw [shift={(7,0)}](0,0) -- (.5,.25)[red][thick];
\draw [shift={(7,0)}](1,0) -- (.5,.25)[red][thick];
\draw [shift={(7,0)}](0,2) -- (.5,1.75)[red][thick];
\draw [shift={(7,0)}](1,2) -- (.5,1.75)[red][thick];
\draw [shift={(7,0)}](2,0) -- (2,2)[red][thick];
\draw [shift={(7,0)}](.5,.25) -- (.5,1.75)[red][very thick];

\draw [shift={(7,0)}](1,2) -- (1.5,2.25)[red][thick];
\draw [shift={(7,0)}](2,2) -- (1.5,2.25)[red][thick];
\draw [shift={(7,0)}](1,4) -- (1.5,3.75)[red][thick];
\draw [shift={(7,0)}](2,4) -- (1.5,3.75)[red][thick];
\draw [shift={(7,0)}](0,2) -- (0,4)[red][thick];
\draw [shift={(7,0)}](1.5,2.25) -- (1.5,3.75)[red][very thick];

\draw [shift={(7,0)}](0,4) -- (.5,4.25)[red][thick];
\draw [shift={(7,0)}](1,4) -- (.5,4.25)[red][thick];
\draw [shift={(7,0)}](0,6) -- (.5,5.75)[red][thick];
\draw [shift={(7,0)}](1,6) -- (.5,5.75)[red][thick];
\draw [shift={(7,0)}](2,4) -- (2,6)[red][thick];
\draw [shift={(7,0)}](.5,4.25) -- (.5,5.75)[red][very thick];
\draw [shift={(7,0)}](-.5,0) .. controls (-.4,1.5) .. (.5,3)[black][thick][dashed];
\draw [shift={(7,0)}](1.5,6) .. controls (1,.5) .. (.5,3)[black][thick][dashed];

\filldraw [red] [shift={(7,0)}](0, 3) circle (2pt);
\draw [shift={(7,0)}](-.5,3) node{$a$};

\end{tikzpicture}
\end{equation}

\end{proof}

\subsection{Main results}
Let $\mathcal{W}=\mathcal{W}(n,1)$ be the monoidal category of $W(n,1)$-bimodules generated by the $W_i$ for $i=1,\ldots,n-1$.  
Recall that $W$ was shorthand notation for $W(n,1)$.
We will define a functor $F \colon \mathcal{H} \rightarrow \mathcal{W}$.

On objects:
\begin{equation*}
F \colon (i_1, \ldots, i_r) \mapsto W_{i_1} \otimes \cdots \otimes W_{i_r}.
\end{equation*}
On morphisms:

\begin{equation*}
\begin{tikzpicture}
\draw (0,.5) node{$F \colon $};
\draw (1,0) -- (1,1)[green][thick];
\draw (.75,.25) node{$i$};
\draw (1.5,.5) node{$\mapsto$};
\draw (3,.5) node{$\Id \colon W_i \rightarrow W_i$};

\end{tikzpicture}
\end{equation*}

\begin{equation*}
\begin{tikzpicture}
\draw (0,.5) node{$F \colon $};
\draw (1,0) -- (1,1)[green][thick];
\draw (.85,.15) node{$i$};
\draw (.75,.5) node{$x_j$};

\draw (1.5,.5) node{$\mapsto$};
\draw (3,.5) node{$x_j . \colon W_i \rightarrow W_i$};

\end{tikzpicture}
\end{equation*}

\begin{equation*}
\begin{tikzpicture}
\draw (0,.5) node{$F \colon $};
\draw (.8,0) -- (.8,1)[green][thick];
\draw (.65,.15) node{$i$};
\draw (1.05,.5) node{$x_j$};

\draw (1.5,.5) node{$\mapsto$};
\draw (3,.5) node{$.x_j \colon W_i \rightarrow W_i$};

\end{tikzpicture}
\end{equation*}

\begin{equation*}
\begin{tikzpicture}
\draw (0,.5) node{$F \colon $};
\draw (1,0) -- (1,.5)[green][thick];
\filldraw [green] (1, .5) circle (2pt);
\draw (.75,.25) node{$i$};
\draw (1.5,.5) node{$\mapsto$};
\draw (3,.5) node{$\epsilon_i \colon W_i \rightarrow {\bf 1}$};

\end{tikzpicture}
\end{equation*}

\begin{equation*}
\begin{tikzpicture}
\draw (0,.5) node{$F \colon $};
\draw (1,1) -- (1,.5)[green][thick];
\filldraw [green] (1, .5) circle (2pt);
\draw (.75,.75) node{$i$};
\draw (1.5,.5) node{$\mapsto$};
\draw (3,.5) node{$\iota_i \colon {\bf 1} \rightarrow W_i$};

\end{tikzpicture}
\end{equation*}

\begin{equation*}
\begin{tikzpicture}
\draw (-.5,.5) node{$F \colon $};

\draw (.5,1) -- (1,.5)[green][thick];
\draw (1.5,1) -- (1,.5)[green][thick];
\draw (1,.5) -- (1,0)[green][thick];
\draw (.8,.25) node{$i$};
\draw (1.55,.85) node{$i$};
\draw (.4,.85) node{$i$};

\draw (2,.5) node{$\mapsto$};
\draw (4,.5) node{\ding{116}$_i$ $\colon W_i \rightarrow W_i \otimes W_i$};

\end{tikzpicture}
\end{equation*}

\begin{equation*}
\begin{tikzpicture}
\draw (-.5,.5) node{$F \colon $};

\draw (.5,0) -- (1,.5)[green][thick];
\draw (1.5,0) -- (1,.5)[green][thick];
\draw (1,.5) -- (1,1)[green][thick];
\draw (.5,.35) node{$i$};
\draw (1.5,.35) node{$i$};
\draw (.85,.85) node{$i$};

\draw (2,.5) node{$\mapsto$};
\draw (4,.5) node{\ding{115}$_i$ $\colon  W_i \otimes W_i \rightarrow W_i$};

\end{tikzpicture}
\end{equation*}

\begin{equation*}
\begin{tikzpicture}
\draw (-.5,.5) node{$F \colon $};

\draw (.5,0) -- (1.5,1)[green][thick];
\draw (1.5,0) -- (.5,1)[green][thick];
\draw (.5,.25) node{$i$};
\draw (1.5,.25) node{$j$};
\draw (.5,.75) node{$j$};
\draw (1.5,.75) node{$i$};

\draw (2.5,.5) node{$\mapsto$};

\draw (5,.5) node{$X_{ij}  \colon  W_i \otimes W_j \rightarrow W_j \otimes W_i$};

\end{tikzpicture}
\end{equation*}

\begin{equation*}
\begin{tikzpicture}

\draw (-.5,.5) node{$F \colon $};

\draw (.5,0) -- (1.5,1)[green][thick];
\draw (1.5,0) -- (.5,1)[green][thick];
\draw (1,0) -- (1,1)[green][thick];
\draw (.5,.25) node{$i$};
\draw (.5,.75) node{$j$};
\draw (1.5,.25) node{$i$};
\draw (1.5,.75) node{$j$};
\draw (.9,.15) node{$j$};
\draw (.9,.9) node{$i$};

\draw (2.5,.5) node{$\mapsto$};

\draw (6,.5) node{ \ding{91}$_{i,j,i}^{j,i,j}$  $ \colon  W_i \otimes W_j \otimes W_i \rightarrow W_j \otimes W_i \otimes W_j$};

\end{tikzpicture}
\end{equation*}

\begin{theorem}
\label{heckeactiontheorem}
The assignment $F \colon \mathcal{H} \rightarrow \mathcal{W}$ is a functor
from the Hecke category $\mathcal{H}$ to the homotopy category of $W(n,1)$-bimodules generated by the $W_i$.
\end{theorem}

\begin{proof}
There are many relations to check on the various generators of the bimodules.
Relations on generators where no black strand is present were checked in ~\cite{EK}.

We will only prove here that relation ~\eqref{relation3.23} holds. 
They are all routine calculations.

The fact that ~\eqref{relation3.23} holds for generators ~\eqref{noblackinterference2} (with $a=0,1$) follows from ~\cite{EK}.  We must only check the relation on the two other generators of the bimodule $ W_i \otimes W_{i+1} \otimes W_i$.
We compute the first term on the right hand side of ~\eqref{relation3.23} on the generator ~\eqref{blackinterference2} with $a=0$.

\begin{equation}
\label{1sttermI}
\begin{tikzpicture}
[scale=0.50]

\draw (0,0) -- (.5,.25)[red][thick];
\draw (1,0) -- (.5,.25)[red][thick];
\draw (0,2) -- (.5,1.75)[red][thick];
\draw (1,2) -- (.5,1.75)[red][thick];
\draw (2,0) -- (2,2)[red][thick];
\draw (.5,.25) -- (.5,1.75)[red][very thick];

\draw (1,2) -- (1.5,2.25)[red][thick];
\draw (2,2) -- (1.5,2.25)[red][thick];
\draw (1,4) -- (1.5,3.75)[red][thick];
\draw (2,4) -- (1.5,3.75)[red][thick];
\draw (0,2) -- (0,4)[red][thick];
\draw (1.5,2.25) -- (1.5,3.75)[red][very thick];

\draw (0,4) -- (.5,4.25)[red][thick];
\draw (1,4) -- (.5,4.25)[red][thick];
\draw (0,6) -- (.5,5.75)[red][thick];
\draw (1,6) -- (.5,5.75)[red][thick];
\draw (2,4) -- (2,6)[red][thick];
\draw (.5,4.25) -- (.5,5.75)[red][very thick];
\draw (1.5,0) .. controls (1,3) .. (1.5,6)[black][thick][dashed];

\draw (3,3) node{$\mapsto$};


\draw (5,0) -- (5.5,.25)[red][thick];
\draw (6,0) -- (5.5,.25)[red][thick];
\draw (5,2) -- (5.5,1.75)[red][thick];
\draw (6,2) -- (5.5,1.75)[red][thick];
\draw (4,0) -- (4,2)[red][thick];
\draw (5.5,.25) -- (5.5,1.75)[red][very thick];

\draw (5,2) -- (4.5,2.25)[red][thick];
\draw (4,2) -- (4.5,2.25)[red][thick];
\draw (5,4) -- (4.5,3.75)[red][thick];
\draw (4,4) -- (4.5,3.75)[red][thick];
\draw (6,2) -- (6,4)[red][thick];
\draw (4.5,2.25) -- (4.5,3.75)[red][very thick];

\draw (5,4) -- (5.5,4.25)[red][thick];
\draw (6,4) -- (5.5,4.25)[red][thick];
\draw (5,6) -- (5.5,5.75)[red][thick];
\draw (6,6) -- (5.5,5.75)[red][thick];
\draw (4,4) -- (4,6)[red][thick];
\draw (5.5,4.25) -- (5.5,5.75)[red][very thick];
\draw (5.6,0) .. controls (6.5,3) .. (5.6,6)[black][thick][dashed];

\draw (7,3) node{$\mapsto$};


\draw [shift={(8,0)}](0,0) -- (.5,.25)[red][thick];
\draw [shift={(8,0)}](1,0) -- (.5,.25)[red][thick];
\draw [shift={(8,0)}](0,2) -- (.5,1.75)[red][thick];
\draw [shift={(8,0)}](1,2) -- (.5,1.75)[red][thick];
\draw [shift={(8,0)}](2,0) -- (2,2)[red][thick];
\draw [shift={(8,0)}](.5,.25) -- (.5,1.75)[red][very thick];

\draw [shift={(8,0)}](1,2) -- (1.5,2.25)[red][thick];
\draw [shift={(8,0)}](2,2) -- (1.5,2.25)[red][thick];
\draw [shift={(8,0)}](1,4) -- (1.5,3.75)[red][thick];
\draw [shift={(8,0)}](2,4) -- (1.5,3.75)[red][thick];
\draw [shift={(8,0)}](0,2) -- (0,4)[red][thick];
\draw [shift={(8,0)}](1.5,2.25) -- (1.5,3.75)[red][very thick];

\draw [shift={(8,0)}](0,4) -- (.5,4.25)[red][thick];
\draw [shift={(8,0)}](1,4) -- (.5,4.25)[red][thick];
\draw [shift={(8,0)}](0,6) -- (.5,5.75)[red][thick];
\draw [shift={(8,0)}](1,6) -- (.5,5.75)[red][thick];
\draw [shift={(8,0)}](2,4) -- (2,6)[red][thick];
\draw [shift={(8,0)}](.5,4.25) -- (.5,5.75)[red][very thick];
\draw [shift={(8,0)}](1.5,0) .. controls (2.5,3) .. (1.5,6)[black][thick][dashed];

\end{tikzpicture}
\end{equation}

Applying four successive maps for the second term in the right hand side of ~\eqref{relation3.23} gives

\begin{equation}
\label{2ndtermI}
\begin{tikzpicture}
[scale=0.50]

\draw (0,0) -- (.5,.25)[red][thick];
\draw (1,0) -- (.5,.25)[red][thick];
\draw (0,2) -- (.5,1.75)[red][thick];
\draw (1,2) -- (.5,1.75)[red][thick];
\draw (2,0) -- (2,2)[red][thick];
\draw (.5,.25) -- (.5,1.75)[red][very thick];

\draw (1,2) -- (1.5,2.25)[red][thick];
\draw (2,2) -- (1.5,2.25)[red][thick];
\draw (1,4) -- (1.5,3.75)[red][thick];
\draw (2,4) -- (1.5,3.75)[red][thick];
\draw (0,2) -- (0,4)[red][thick];
\draw (1.5,2.25) -- (1.5,3.75)[red][very thick];

\draw (0,4) -- (.5,4.25)[red][thick];
\draw (1,4) -- (.5,4.25)[red][thick];
\draw (0,6) -- (.5,5.75)[red][thick];
\draw (1,6) -- (.5,5.75)[red][thick];
\draw (2,4) -- (2,6)[red][thick];
\draw (.5,4.25) -- (.5,5.75)[red][very thick];
\draw (1.5,0) .. controls (1,3) .. (1.5,6)[black][thick][dashed];

\draw (2.9,3) node{$\mapsto$};

\draw (3.5,3) node{$-$};

\draw [shift={(4,0)}](0,0) -- (.5,.25)[red][thick];
\draw [shift={(4,0)}](1,0) -- (.5,.25)[red][thick];
\draw [shift={(4,0)}](0,2) -- (.5,1.75)[red][thick];
\draw [shift={(4,0)}](1,2) -- (.5,1.75)[red][thick];
\draw [shift={(4,0)}](2,0) -- (2,2)[red][thick];
\draw [shift={(4,0)}](.5,.25) -- (.5,1.75)[red][very thick];

\draw [shift={(4,0)}](1,2) -- (1.5,2.25)[red][thick];
\draw [shift={(4,0)}](2,2) -- (1.5,2.25)[red][thick];
\draw [shift={(4,0)}](1,4) -- (1.5,3.75)[red][thick];
\draw [shift={(4,0)}](2,4) -- (1.5,3.75)[red][thick];
\draw [shift={(4,0)}](0,2) -- (0,4)[red][thick];
\draw [shift={(4,0)}](1.5,2.25) -- (1.5,3.75)[red][very thick];

\draw [shift={(4,0)}](0,4) -- (.5,4.25)[red][thick];
\draw [shift={(4,0)}](1,4) -- (.5,4.25)[red][thick];
\draw [shift={(4,0)}](0,6) -- (.5,5.75)[red][thick];
\draw [shift={(4,0)}](1,6) -- (.5,5.75)[red][thick];
\draw [shift={(4,0)}](2,4) -- (2,6)[red][thick];
\draw [shift={(4,0)}](.5,4.25) -- (.5,5.75)[red][very thick];
\draw [shift={(4,0)}](1.5,0) .. controls (1,3) .. (1.5,6)[black][thick][dashed];

\draw (7,3) node{$+$};


\draw [shift={(8,0)}](0,0) -- (.5,.25)[red][thick];
\draw [shift={(8,0)}](1,0) -- (.5,.25)[red][thick];
\draw [shift={(8,0)}](0,2) -- (.5,1.75)[red][thick];
\draw [shift={(8,0)}](1,2) -- (.5,1.75)[red][thick];
\draw [shift={(8,0)}](2,0) -- (2,2)[red][thick];
\draw [shift={(8,0)}](.5,.25) -- (.5,1.75)[red][very thick];

\draw [shift={(8,0)}](1,2) -- (1.5,2.25)[red][thick];
\draw [shift={(8,0)}](2,2) -- (1.5,2.25)[red][thick];
\draw [shift={(8,0)}](1,4) -- (1.5,3.75)[red][thick];
\draw [shift={(8,0)}](2,4) -- (1.5,3.75)[red][thick];
\draw [shift={(8,0)}](0,2) -- (0,4)[red][thick];
\draw [shift={(8,0)}](1.5,2.25) -- (1.5,3.75)[red][very thick];

\draw [shift={(8,0)}](0,4) -- (.5,4.25)[red][thick];
\draw [shift={(8,0)}](1,4) -- (.5,4.25)[red][thick];
\draw [shift={(8,0)}](0,6) -- (.5,5.75)[red][thick];
\draw [shift={(8,0)}](1,6) -- (.5,5.75)[red][thick];
\draw [shift={(8,0)}](2,4) -- (2,6)[red][thick];
\draw [shift={(8,0)}](.5,4.25) -- (.5,5.75)[red][very thick];
\draw [shift={(8,0)}](1.5,0) .. controls (2.5,3) .. (1.5,6)[black][thick][dashed];

\end{tikzpicture}
\end{equation}

The difference of the maps in ~\eqref{1sttermI} and ~\eqref{2ndtermI} is simply the identity map on the generator ~\eqref{blackinterference2} for $a=0$ thus establishing ~\eqref{relation3.23} for this generator.

We compute the first term on the right hand side of ~\eqref{relation3.23} on the generator ~\eqref{blackinterference2} with $a=1$.

\begin{equation}
\label{1sttermII}
\begin{tikzpicture}
[scale=0.50]

\draw (0,0) -- (.5,.25)[red][thick];
\draw (1,0) -- (.5,.25)[red][thick];
\draw (0,2) -- (.5,1.75)[red][thick];
\draw (1,2) -- (.5,1.75)[red][thick];
\draw (2,0) -- (2,2)[red][thick];
\draw (.5,.25) -- (.5,1.75)[red][very thick];

\draw (1,2) -- (1.5,2.25)[red][thick];
\draw (2,2) -- (1.5,2.25)[red][thick];
\draw (1,4) -- (1.5,3.75)[red][thick];
\draw (2,4) -- (1.5,3.75)[red][thick];
\draw (0,2) -- (0,4)[red][thick];
\draw (1.5,2.25) -- (1.5,3.75)[red][very thick];

\draw (0,4) -- (.5,4.25)[red][thick];
\draw (1,4) -- (.5,4.25)[red][thick];
\draw (0,6) -- (.5,5.75)[red][thick];
\draw (1,6) -- (.5,5.75)[red][thick];
\draw (2,4) -- (2,6)[red][thick];
\draw (.5,4.25) -- (.5,5.75)[red][very thick];
\draw (1.5,0) .. controls (1,3) .. (1.5,6)[black][thick][dashed];

\filldraw [red] (0,3) circle (2pt);

\draw (3,3) node{$\mapsto$};


\draw [shift={(4,0)}](0,0) -- (.5,.25)[red][thick];
\draw  [shift={(4,0)}](1,0) -- (.5,.25)[red][thick];
\draw  [shift={(4,0)}](0,2) -- (.5,1.75)[red][thick];
\draw  [shift={(4,0)}](1,2) -- (.5,1.75)[red][thick];
\draw  [shift={(4,0)}](2,0) -- (2,2)[red][thick];
\draw  [shift={(4,0)}](.5,.25) -- (.5,1.75)[red][very thick];

\draw  [shift={(4,0)}](1,2) -- (1.5,2.25)[red][thick];
\draw  [shift={(4,0)}](2,2) -- (1.5,2.25)[red][thick];
\draw  [shift={(4,0)}](1,4) -- (1.5,3.75)[red][thick];
\draw  [shift={(4,0)}](2,4) -- (1.5,3.75)[red][thick];
\draw  [shift={(4,0)}](0,2) -- (0,4)[red][thick];
\draw  [shift={(4,0)}](1.5,2.25) -- (1.5,3.75)[red][very thick];

\draw  [shift={(4,0)}](0,4) -- (.5,4.25)[red][thick];
\draw  [shift={(4,0)}](1,4) -- (.5,4.25)[red][thick];
\draw  [shift={(4,0)}](0,6) -- (.5,5.75)[red][thick];
\draw  [shift={(4,0)}](1,6) -- (.5,5.75)[red][thick];
\draw  [shift={(4,0)}](2,4) -- (2,6)[red][thick];
\draw  [shift={(4,0)}](.5,4.25) -- (.5,5.75)[red][very thick];
\draw  [shift={(4,0)}](1.5,0) .. controls (1,3) .. (1.5,6)[black][thick][dashed];

\filldraw  [shift={(4,0)}][red] (0,3) circle (2pt);

\draw [shift={(4,0)}](3,3) node{$-$};


\draw [shift={(8,0)}](0,0) -- (.5,.25)[red][thick];
\draw  [shift={(8,0)}](1,0) -- (.5,.25)[red][thick];
\draw  [shift={(8,0)}](0,2) -- (.5,1.75)[red][thick];
\draw  [shift={(8,0)}](1,2) -- (.5,1.75)[red][thick];
\draw  [shift={(8,0)}](2,0) -- (2,2)[red][thick];
\draw  [shift={(8,0)}](.5,.25) -- (.5,1.75)[red][very thick];

\draw  [shift={(8,0)}](1,2) -- (1.5,2.25)[red][thick];
\draw  [shift={(8,0)}](2,2) -- (1.5,2.25)[red][thick];
\draw  [shift={(8,0)}](1,4) -- (1.5,3.75)[red][thick];
\draw  [shift={(8,0)}](2,4) -- (1.5,3.75)[red][thick];
\draw  [shift={(8,0)}](0,2) -- (0,4)[red][thick];
\draw  [shift={(8,0)}](1.5,2.25) -- (1.5,3.75)[red][very thick];

\draw  [shift={(8,0)}](0,4) -- (.5,4.25)[red][thick];
\draw  [shift={(8,0)}](1,4) -- (.5,4.25)[red][thick];
\draw  [shift={(8,0)}](0,6) -- (.5,5.75)[red][thick];
\draw  [shift={(8,0)}](1,6) -- (.5,5.75)[red][thick];
\draw  [shift={(8,0)}](2,4) -- (2,6)[red][thick];
\draw  [shift={(8,0)}](.5,4.25) -- (.5,5.75)[red][very thick];
\draw  [shift={(8,0)}](1.5,0) .. controls (1,3) .. (1.5,6)[black][thick][dashed];

\filldraw  [shift={(8,0)}][black] (1.5,6) circle (2pt);

\draw [shift={(8,0)}](3,3) node{$+$};


\draw [shift={(12,0)}](0,0) -- (.5,.25)[red][thick];
\draw  [shift={(12,0)}](1,0) -- (.5,.25)[red][thick];
\draw  [shift={(12,0)}](0,2) -- (.5,1.75)[red][thick];
\draw  [shift={(12,0)}](1,2) -- (.5,1.75)[red][thick];
\draw  [shift={(12,0)}](2,0) -- (2,2)[red][thick];
\draw  [shift={(12,0)}](.5,.25) -- (.5,1.75)[red][very thick];

\draw  [shift={(12,0)}](1,2) -- (1.5,2.25)[red][thick];
\draw  [shift={(12,0)}](2,2) -- (1.5,2.25)[red][thick];
\draw  [shift={(12,0)}](1,4) -- (1.5,3.75)[red][thick];
\draw  [shift={(12,0)}](2,4) -- (1.5,3.75)[red][thick];
\draw  [shift={(12,0)}](0,2) -- (0,4)[red][thick];
\draw  [shift={(12,0)}](1.5,2.25) -- (1.5,3.75)[red][very thick];

\draw  [shift={(12,0)}](0,4) -- (.5,4.25)[red][thick];
\draw  [shift={(12,0)}](1,4) -- (.5,4.25)[red][thick];
\draw  [shift={(12,0)}](0,6) -- (.5,5.75)[red][thick];
\draw  [shift={(12,0)}](1,6) -- (.5,5.75)[red][thick];
\draw  [shift={(12,0)}](2,4) -- (2,6)[red][thick];
\draw  [shift={(12,0)}](.5,4.25) -- (.5,5.75)[red][very thick];
\draw  [shift={(12,0)}](1.5,0) .. controls (2.5,3) .. (1.5,6)[black][thick][dashed];

\filldraw  [shift={(12,0)}][black] (1.5,6) circle (2pt);

\end{tikzpicture}
\end{equation}

Applying four successive maps for the second term in the right hand side of ~\eqref{relation3.23} gives

\begin{equation}
\label{2ndtermII}
\begin{tikzpicture}
[scale=0.50]

\draw (0,0) -- (.5,.25)[red][thick];
\draw (1,0) -- (.5,.25)[red][thick];
\draw (0,2) -- (.5,1.75)[red][thick];
\draw (1,2) -- (.5,1.75)[red][thick];
\draw (2,0) -- (2,2)[red][thick];
\draw (.5,.25) -- (.5,1.75)[red][very thick];

\draw (1,2) -- (1.5,2.25)[red][thick];
\draw (2,2) -- (1.5,2.25)[red][thick];
\draw (1,4) -- (1.5,3.75)[red][thick];
\draw (2,4) -- (1.5,3.75)[red][thick];
\draw (0,2) -- (0,4)[red][thick];
\draw (1.5,2.25) -- (1.5,3.75)[red][very thick];

\draw (0,4) -- (.5,4.25)[red][thick];
\draw (1,4) -- (.5,4.25)[red][thick];
\draw (0,6) -- (.5,5.75)[red][thick];
\draw (1,6) -- (.5,5.75)[red][thick];
\draw (2,4) -- (2,6)[red][thick];
\draw (.5,4.25) -- (.5,5.75)[red][very thick];
\draw (1.5,0) .. controls (1,3) .. (1.5,6)[black][thick][dashed];

\filldraw [red] (0,3) circle (2pt);

\draw (3,3) node{$\mapsto$};


\draw [shift={(4,0)}](3,3) node{$-$};


\draw [shift={(8,0)}](0,0) -- (.5,.25)[red][thick];
\draw  [shift={(8,0)}](1,0) -- (.5,.25)[red][thick];
\draw  [shift={(8,0)}](0,2) -- (.5,1.75)[red][thick];
\draw  [shift={(8,0)}](1,2) -- (.5,1.75)[red][thick];
\draw  [shift={(8,0)}](2,0) -- (2,2)[red][thick];
\draw  [shift={(8,0)}](.5,.25) -- (.5,1.75)[red][very thick];

\draw  [shift={(8,0)}](1,2) -- (1.5,2.25)[red][thick];
\draw  [shift={(8,0)}](2,2) -- (1.5,2.25)[red][thick];
\draw  [shift={(8,0)}](1,4) -- (1.5,3.75)[red][thick];
\draw  [shift={(8,0)}](2,4) -- (1.5,3.75)[red][thick];
\draw  [shift={(8,0)}](0,2) -- (0,4)[red][thick];
\draw  [shift={(8,0)}](1.5,2.25) -- (1.5,3.75)[red][very thick];

\draw  [shift={(8,0)}](0,4) -- (.5,4.25)[red][thick];
\draw  [shift={(8,0)}](1,4) -- (.5,4.25)[red][thick];
\draw  [shift={(8,0)}](0,6) -- (.5,5.75)[red][thick];
\draw  [shift={(8,0)}](1,6) -- (.5,5.75)[red][thick];
\draw  [shift={(8,0)}](2,4) -- (2,6)[red][thick];
\draw  [shift={(8,0)}](.5,4.25) -- (.5,5.75)[red][very thick];
\draw  [shift={(8,0)}](1.5,0) .. controls (1,3) .. (1.5,6)[black][thick][dashed];

\filldraw  [shift={(8,0)}][black] (1.5,6) circle (2pt);



\draw [shift={(4,0)}](0,0) -- (.5,.25)[red][thick];
\draw  [shift={(4,0)}](1,0) -- (.5,.25)[red][thick];
\draw  [shift={(4,0)}](0,2) -- (.5,1.75)[red][thick];
\draw  [shift={(4,0)}](1,2) -- (.5,1.75)[red][thick];
\draw  [shift={(4,0)}](2,0) -- (2,2)[red][thick];
\draw  [shift={(4,0)}](.5,.25) -- (.5,1.75)[red][very thick];

\draw  [shift={(4,0)}](1,2) -- (1.5,2.25)[red][thick];
\draw  [shift={(4,0)}](2,2) -- (1.5,2.25)[red][thick];
\draw  [shift={(4,0)}](1,4) -- (1.5,3.75)[red][thick];
\draw  [shift={(4,0)}](2,4) -- (1.5,3.75)[red][thick];
\draw  [shift={(4,0)}](0,2) -- (0,4)[red][thick];
\draw  [shift={(4,0)}](1.5,2.25) -- (1.5,3.75)[red][very thick];

\draw  [shift={(4,0)}](0,4) -- (.5,4.25)[red][thick];
\draw  [shift={(4,0)}](1,4) -- (.5,4.25)[red][thick];
\draw  [shift={(4,0)}](0,6) -- (.5,5.75)[red][thick];
\draw  [shift={(4,0)}](1,6) -- (.5,5.75)[red][thick];
\draw  [shift={(4,0)}](2,4) -- (2,6)[red][thick];
\draw  [shift={(4,0)}](.5,4.25) -- (.5,5.75)[red][very thick];
\draw  [shift={(4,0)}](1.5,0) .. controls (2.5,3) .. (1.5,6)[black][thick][dashed];

\filldraw  [shift={(4,0)}][black] (1.5,6) circle (2pt);

\end{tikzpicture}
\end{equation}

The difference of the maps in ~\eqref{1sttermII} and ~\eqref{2ndtermII} is simply the identity map on the generator ~\eqref{blackinterference2} for $a=1$ thus establishing ~\eqref{relation3.23} for this generator.
\end{proof}

We now define the analogue of Rouquier complexes for the rings $W(n,1)$:
\[
\xymatrix{
\sigma_i := & 0 \ar[r] & W(n,1) \ar[r]^{\iota_i} & W_i \langle -1 \rangle \ar[r] & 0
}
\]
\[
\xymatrix{
\sigma_i^{-1} := & 0 \ar[r] & W_i \langle 1 \rangle \ar[r]^{\epsilon_i \hspace{.1in}} & W(n,1) \ar[r] & 0
}
\]
where in both complexes the object $W_i$ is in cohomological degree zero.

\begin{corollary}
\label{braidcorollarysigma}
As endofunctors of the homotopy category $Kom(W(n,1)\gpmod)$, the complexes $\sigma_i $ and $\sigma_i^{-1}$ satisfy the following isomorphisms
\begin{enumerate}
\item $ \sigma_i \circ \sigma_i^{-1} \cong \Id \cong \sigma_i^{-1} \circ \sigma_i $
\item $ \sigma_i \circ \sigma_j \cong \sigma_j \circ \sigma_i$ if $|i-j|>1$
\item $ \sigma_i \circ \sigma_{i+1} \circ \sigma_i \cong \sigma_{i+1} \circ \sigma_i \circ \sigma_{i+1}$.
\end{enumerate}
Furthermore there is a functor from the category of braid cobordisms to the category above.
\end{corollary}

\begin{proof}
Theorem ~\ref{heckeactiontheorem} and ~\cite[Theorem 1]{EKr} imply the corollary.
\end{proof}

\begin{remark}
The proof of Theorem ~\ref{heckeactiontheorem} did not depend upon the cyclotomic relation in $W(n,1)$ and so the corresponding statement is also true for $\widetilde{W}(n,1)$.
\end{remark}

\subsection{Braid group action on ${}^{\emptyset} \mathcal{R}^{\hat{1}}$}
There is a braid group action on the homotopy category of singular Soergel bimodules ${}^{\emptyset} \mathcal{R}^{\hat{k}}$
coming from the normal braid group action on the homotopy category of $\mathcal{R}$.  

Define functors
\begin{equation*}
T_i \colon Kom({}^{\emptyset} \mathcal{R}^{\hat{1}}) \rightarrow Kom({}^{\emptyset} \mathcal{R}^{\hat{1}}) \hspace{.5in} T_i(M) = (B_i \langle 1 \rangle \longrightarrow R) \otimes_R M
\end{equation*}
where the map $B_i \rightarrow R$ is just multiplication.

Define functors
\begin{equation*}
T_i^{-1} \colon Kom({}^{\emptyset} \mathcal{R}^{\hat{1}}) \rightarrow Kom({}^{\emptyset} \mathcal{R}^{\hat{1}}) \hspace{.5in} T_i^{-1}(M) = (R  \longrightarrow B_i \langle -1 \rangle) \otimes_R M
\end{equation*}
where the map $R \rightarrow B_i \langle -1 \rangle$ is determined by $ 1 \mapsto x_i \otimes 1 -1 \otimes x_{i+1}$.

The following theorem is an immediate consequence of ~\cite[Theorem 3.5]{Ro2}.
\begin{theorem}
\label{braidtheoremregappliedtosing}
The functors $T_i$ for $i=1,\ldots,n-1$ satisfy braid group relations
\begin{enumerate}
\item $T_i \circ T_i^{-1} \cong \Id \cong T_i^{-1} \circ T_i$,
\item $T_i \circ T_j \cong T_j \circ T_i$ if $|i-j|>1$,
\item $T_i \circ T_{i+1} \circ T_i \cong T_{i+1} \circ T_i \circ T_{i+1}$.
\end{enumerate}
The functors $T_i, T_i^{-1}$ descend to an action on $K_0(Kom({}^{\emptyset} \mathcal{R}^{\hat{1}}))$ endowing the Grothendieck group with the structure of the Burau representation of the braid group.
\end{theorem}

\vspace{.25in}

\section{References} 


\def\refname{}

\end{document}